\documentclass[11pt,reqno,intlimits]{amsart}
\usepackage{amsfonts,amssymb,amsmath,amsthm}
\usepackage{graphicx}
\usepackage[colorlinks=true,urlcolor=blue,citecolor=red,linkcolor=blue,linktocpage,pdfpagelabels,bookmarksnumbered,bookmarksopen]{hyperref}
\begin{document}

\newtheorem*{theorem*}{Theorem}
\newtheorem{theorem}{Theorem}[section]
\newtheorem{lemma}[theorem]{Lemma}
\newtheorem{corollary}[theorem]{Corollary}
\newtheorem{proposition}[theorem]{Proposition}

\theoremstyle{definition}
\newtheorem{definition}[theorem]{Definition}
\newtheorem{example}[theorem]{Example}

\theoremstyle{remark}
\newtheorem{remark}[theorem]{Remark}
\numberwithin{equation}{section}

\newcommand{\nor}{\Arrowvert}
\newcommand{\field}[1]{\mathbb{#1}}
\newcommand{\R}{\field{R}}
\newcommand{\N}{\field{N}}
\def\l{{\lambda}}
\def\d{{\delta}}
\def\a{{\alpha}}
\def\g{{\gamma}}
\def\de{{\partial}}
\def\L{{\Lambda}}
\def\Om{{\Omega}}
\def\e{{\varepsilon}}
\providecommand{\abs}[1]{\lvert#1\rvert}
\providecommand{\norm}[1]{\lVert#1\rVert}

\title[Symmetry breaking and Morse index..]{Symmetry breaking and Morse index of solutions of nonlinear elliptic problems in the plane}
\thanks{The first two authors are supported by PRIN-2009-WRJ3W7 grant,
S. Neves has been partially supported by FAPESP}

\author[Gladiali]{Francesca  Gladiali}
\address{Dipartimento Polcoming-Matematica e Fisica, Universit\`a  di Sassari  - Via Piandanna 4, 07100 Sassari - Italy.}
\email{fgladiali@uniss.it}
\author[Grossi]{Massimo Grossi}
\address{Dipartimento di Matematica, Universit\`a di Roma
La Sapienza, P.le A. Moro 2 - 00185 Roma- Italy.}
\email{massimo.grossi@uniroma1.it}
\author[Neves]{S\'ergio L. N. Neves}
\address{Departamento de Matem\'atica, Universidade Estadual de Campinas - IMECC,
Rua S\'ergio Buarque de Holanda 651, Campinas-SP 13083-859 - Brazil.}
\email{sergio184@gmail.com}

\begin{abstract}
In this paper we study the problem
\begin{equation}
\begin{cases}\label{i0} \tag{$\mathcal{P}$}
-\Delta u =\left(\frac{2+\alpha}{2}\right)^2\abs{x}^{\alpha}f(\lambda ,u), & \hbox{ in }B_1 \\
u > 0, & \hbox{ in }B_1 \\
 u = 0, & \hbox{ on } \partial B_1
\end{cases}
\end{equation}
where $B_1$ is the unit ball of $\R^2$, $f$ is a smooth
nonlinearity and $\a$, $\l$ are real numbers with $\a>0$. From a
careful study of the linearized operator we compute the Morse
index of some radial solutions to \eqref{i0}. Moreover, using the
bifurcation theory, we prove the existence of branches of nonradial solutions for suitable values of the positive parameter
$\l$. The case $f(\lambda ,u)=\l e^u$ provides more detailed
information.
\end{abstract}
\maketitle

\section{Introduction and main results}
In this paper we study the problem
\begin{equation}\label{a1}
\begin{cases}
-\Delta u = \left(\frac{2+\alpha}{2}\right)^2\abs{x}^{\alpha}f(\lambda ,u), & \hbox{ in }B_1 \\
u > 0, & \hbox{ in }B_1 \\
 u = 0, & \hbox{ on } \partial B_1
\end{cases}
\end{equation}
where $\a,\l$ are real numbers, $\alpha>0$ and $B_1$ is the unit ball of $\R^2$.\\
The nonlinearity $f(t,s)$ satisfies
\begin{equation}\label{j1}
f:(a,b)\times\R^+\rightarrow\R^+,\ a,b\in\R,\quad f\in C^{1,1}_{loc}((a,b)\times[0,+\infty)).
\end{equation}

Problem \eqref{a1} models different type of equations. When $f(\l,s)=s^\l$
with $\l>1$ it is known as H\'enon problem and arises in the study of stellar clusters.
Other interesting examples are given by $f(\l,s)=\l(1+s)^p$ with $p>1$, $\l>0$ and $f(\l,s)=\l e^s$. In this last case problem \eqref{a1} is sometimes known as Liouville equation and arises in the study of vortices of Euler flows in the gauge field theory, when the equation involves singular sources. \\
Observe that the presence of the term $|x|^\a$ allows the existence of nonradial solutions for \eqref{a1} and the constant $\left(\frac{2+\a}2\right)^2$ in \eqref{a1}, in many cases, can be merged into the equation.\\
In the particular case of $f(\l,s)=\l e^s$ problem \eqref{a1} has been studied by many authors.
Del Pino, Kowalczyk and Musso in \cite{DKM} proved the existence of solutions concentrating at
one or more points when $\a$ is fixed and $\l$ is close to zero
 (see also \cite{DAP} for a generalization of this result).
We quote also the results in \cite{B}, \cite{BCLT} and \cite{CL1} where the behavior
of the solutions as $\l\to 0$ is considered. \\
In this paper we are interested in studying existence of nonradial
solutions to \eqref{a1}.  The main tool to get our results will be
the {\em bifurcation theory}. We want to stress that the results
we obtain, both for the general problem \eqref{a1} that for the
special case of $f(\l,s)=\l e^s$, allow $\a$ and $\l$ to vary in
all the range of existence of solutions and not only in a
neighborhood of some specific value (as $\l=0$).

To our knowledge one of the few results where $\l$ varies in all
its range is due to Suzuki \cite{Su}, where he proved the
nondegeneracy of the solution in simply connected domains when
$\a=0$, $f(\l, s) =\l e^s$, $\l\int_\Omega e^u<8\pi$ and
$\l\in(0,\l^*)$. Here $\l^*$ is the maximal value of $\l$ such
that \eqref{a1} has a solution.
 In a similar spirit we will obtain some existence results of solutions to
\eqref{a1}.\\
A crucial role in our analysis is given by the autonomous problem
associated to \eqref{a1},
\begin{equation}\label{i2}
\begin{cases}
-\Delta v =f(\lambda ,v), & \hbox{ in }B_1 \\
v > 0, & \hbox{ in }B_1 \\
 v = 0, & \hbox{ on } \partial B_1.
\end{cases}
\end{equation}
The well known result by Gidas, Ni and Nirenberg (\cite{GNN}) tell us that all
solutions to \eqref{i2} are radial (this is no longer true for
\eqref{a1}).\\

Let us remark that, if we restrict ourselves to considering only
radial solutions, then problems  \eqref{a1} and  \eqref{i2} are
equivalent. This can be easily seen using the map $r\mapsto
r^\frac{2+\a}2$ as pointed out in \cite{CG}
(see also \cite{GGN}, \cite{SSW02} and the proof of Proposition  \ref{a4} ).\\
Let us now turn to illustrate the strategy of our paper: first of
all observe that, in many situations, it is possible to establish
the existence of radial solutions $v_\l$ to \eqref{i2} for some
suitable values $\l\in(a,b)$.\\
Then we get that
\begin{equation} \label{f1}
u_{\l,\a}(r)=v_\l(r^\frac{2+\a}2)
\end{equation}
are radial
solutions to \eqref{a1} for any $\a>0$ and $\l\in(a,b)$. If we look for nonradial solutions which
bifurcate from $u_{\l,\a}$, by the implicit function theorem, the
values of $\a$ and $\l$ must satisfy the degeneracy condition,
\begin{equation}\label{a7}
\begin{cases}
-\Delta w-\left(\frac{2+\alpha}{2}\right)^2\abs{x}^{\alpha}f'(\l,u_{\l,\a})w=0, & \hbox{ in }B_1 \\
 w= 0, & \hbox{ on } \partial B_1
\end{cases}
\end{equation}
for some nontrivial $w\in H^1_0(B_1)$, where $f'(\l,u)=\frac{\partial f}{\partial u}(\l,u)$.

In general the computation of the values $\a$ and $\l$ for which
$u_{\l,\a}$ is degenerate is a very difficult problem. However, in
the case where the solution $v_\l$ of \eqref{i2} has Morse index
$1$, we will be able to characterize them. In the following we set $H$ as the closure of $C^{\infty}_0(0,1)$ with respect to the norm $\nor \eta \nor_{\mathcal{H}}=\int_0^1\left((\eta')^2+ \frac{\eta^2}{r^2}\right)rdr$.
\begin{proposition}\label{a4}
Assume \eqref{j1} and that problem \eqref{i2} has a solution $v_{\lambda}$ for
any $\lambda\in(a,b)$. Then problem \eqref{a1} has a radial
solution
$u_{\lambda,\alpha}(r)=v_{\lambda}\left(r^\frac{2+\a}2\right)$ for
any $\lambda\in(a,b)$ and
$\alpha>0$.\\
Moreover if $v_{\lambda}$ is radially nondegenerate and it has
Morse index $1$, setting  
\begin{equation}\label{a5}
\nu_1(\lambda):=\inf_{\eta\in{\mathcal{H}}\atop \eta\not\equiv0,\
}\frac{\int_0^1r(\eta')^2\,
dr-\int_0^1rf'(\lambda,v)\eta^2\, dr}{\int_0^1 r^{-1}\eta^2 \, dr}
\end{equation}
we have that $u_{\lambda,\alpha}$ is degenerate if and only if
$\l$ and $\a$ satisfy
\begin{equation} \label{a6}
\nu_1(\lambda)=-\frac{4k^2}{(2+\alpha)^2}
\end{equation}
for some integer $k\geq 1$. The solutions of \eqref{a7} corresponding to the values of $\l$
and $\a$ which satisfy \eqref{a6} are given by, in polar
coordinates,
\begin{equation}\label{a6a}
\psi_{\l,\a}(r,\theta)=\widetilde\psi_{1,\l}\left(r^\frac{2+\a}2\right)\left[A\sin
\left(\frac{2+\a}2\sqrt{-\nu_1(\l)}\right)\theta+B\cos \left(\frac{2+\a}2\sqrt{-\nu_1(\l)}\right)
\theta\right]
\end{equation}
for any real constant $A$ and $B$ where $\widetilde\psi_{1,\l}$ is
the function which achieves \eqref{a5}.\\
Else if $v_{\lambda}$ is degenerate for some $\hat\l\in(a,b)$,
with eigenfunction $\tilde \psi_{\hat\l}$, then
$u_{\hat\lambda,\alpha}$ is radially degenerate for any $\a>0$
with eigenfunction $\psi_{\hat\l,\a}(r)=\tilde
\psi_{\hat\l}(r^{\frac{2+\a}2})$.
\end{proposition}
\begin{remark}\label{rem1}
Due to the lack of the Hardy inequality in $\R^2$, it is not trivial that the infimum in \eqref{a5} is achieved. However, since the Morse index of $v_\l$ is $1$ then we get that
\begin{equation*}
\nu_1(\lambda)<0.
\end{equation*}
Then Proposition  \ref{B1} applies and so the  the infimum in \eqref{a5} is achieved. Note the the condition on the Morse index of $v_\l$ is crucial. Indeed, if $\nu_1(\lambda)=0$, in Proposition \ref{P1}  is provided an example where $\nu_1(\lambda)$ is not attained.
\end{remark}

Equation \eqref{a6} characterizes all the degeneracy points of the
radial solutions $u_{\lambda,\alpha}$. It seems important to emphasize
that the map $r\mapsto r^{\frac{2+\a}2}$, that links the radial
solutions of \eqref{a1} and \eqref{i2}, allows us to say more: it
identifies the degeneracy points in terms of $\nu_1(\l)$ that is
not directly related
with problem \eqref{a1} but just with problem \eqref{i2}.\\
In addition, the quantity $\nu_1(\l) $, in some specific cases,
can be explicitly computed
 as we shall see, for example, in Proposition \ref{p-ni1}.\\
It seems to us that this link has never been highlighted before
 and can bring useful information such as the calculation of the Morse index.\\
The degeneracy result of Proposition \ref{a4} can be generalized
also to solutions $v_\l$ with Morse index greater than $1$. This
case, however, goes beyond this work and we do not treat it.\\
A first consequence of Proposition \ref{a4} is the computation of
the Morse index of the solution $u_{\lambda,\alpha}$.
\begin{theorem}\label{a11}
Let $v_\l$ be a solution of \eqref{i2} with Morse index 1 and $u_{\l,\a}=v_\l(r^{\frac {2+\a}2})$.
Then the Morse index
$m(\lambda,\alpha)$ of the radial solution $u_{\lambda,\a}$ to
\eqref{a1} (see \eqref{f1}) is equal to
\begin{equation}\label{morse-general}
m(\l,\a)=
\begin{cases}
1+ 2\left[ \frac{\alpha+2}{2} \sqrt{-\nu_1(\lambda)}\right] &\hbox{ if } \,\,\, \frac{\alpha+2}{2}\sqrt{-\nu_1(\lambda)} \not\in \N \\
 (\alpha+2) \sqrt{-\nu_1(\lambda)}  -1 &\hbox{ if } \,\,\, \frac{\alpha+2}2 \sqrt{-\nu_1(\lambda)} \in \N
\end{cases}
\end{equation}
where $[x]$ denotes the greatest integer less than or equal to
$x$.\\
Moreover $m(\l,\a)\rightarrow+\infty$ as $\a\rightarrow+\infty$.
\end{theorem}
Observe that starting from a solution $v_\l$ of \eqref{i2} of
Morse index $1$ we get a radial solution $u_{\l,\a}$ of \eqref{a1}
which is degenerate in a certain number of points and whose Morse
index increases with alpha.\\
Next result shows that, under some additional assumptions, the
condition \eqref{a6} is also sufficient to get solutions which
bifurcate from the radial one.
\begin{theorem}\label{a15}
Suppose that $f$ satisfies \eqref{j1}, $f(\l,0)\geq 0$ and assume that problem \eqref{i2}
has a solution $v_{\lambda}$ for any $\lambda\in(a,b)$ which is
radially nondegenerate with
Morse index $1$.\\ 
Let $\a>0$ be fixed and set
\begin{equation*}
F_k(\l)=\nu_1(\lambda)+\frac{4k^2}{(2+\alpha)^2}
\end{equation*}
for any integer $k\ge1$. If
\begin{itemize}\label{a16}
\item[i)] there exists, for some $k\ge1$ a real value $\lambda=\lambda(k)\in (a,b)$ at which the
function $F_k(\l)$ changes sign in a neighborhood of $\l=\lambda(k)$;
\item[ii)] for any $k\ge1$
the zeros of $F_k(\l)$ in a neighborhood of
$\lambda(k)$ are isolated,
\end{itemize}
then there exists a branch of nonradial solutions of \eqref{a1}
bifurcating from $(\lambda(k),u_{\lambda(k),\a})$.\\
 Moreover branches of bifurcating solutions related to
different values of $k$ are separated and one of the following alternative holds
\begin{itemize}
\item[-] they are unbounded in $(a,b)\times  C^{1,\g}_0(\bar B_1)$
\item[-] they intersect the curve of radial solutions in another bifurcation point related to the same value of $k$,
\item[-] they meet the boundary of $(a,b)\times  C^{1,\g}_0(\bar B_1)$.
\end{itemize}
\end{theorem}
In the previous theorem a crucial role is played by the curves
\begin{equation}\label{b18}
\gamma_k=\left\{(\l,\a)\in(a,b)\times(0,+\infty)\hbox{ such that
}\nu_1(\lambda)+\frac{4k^2}{(2+\alpha)^2}=0\right\}.
\end{equation}
For any integer $k$ greater than $1$ we have that each $\gamma_k$
is a smooth curve of $\R^2$. Theorem \ref{a15} says that there
exists a set ${\mathcal U}_k$ containing $\gamma_k$ in
$(a,b)\times(0,+\infty)$ such that if $(\l,\a)\in{\mathcal U}_k$
then the problem \eqref{a1} has at least one nonradial solution
(see figure \ref{fig1}).

\begin{figure}[t!]
\includegraphics[scale=0.6]{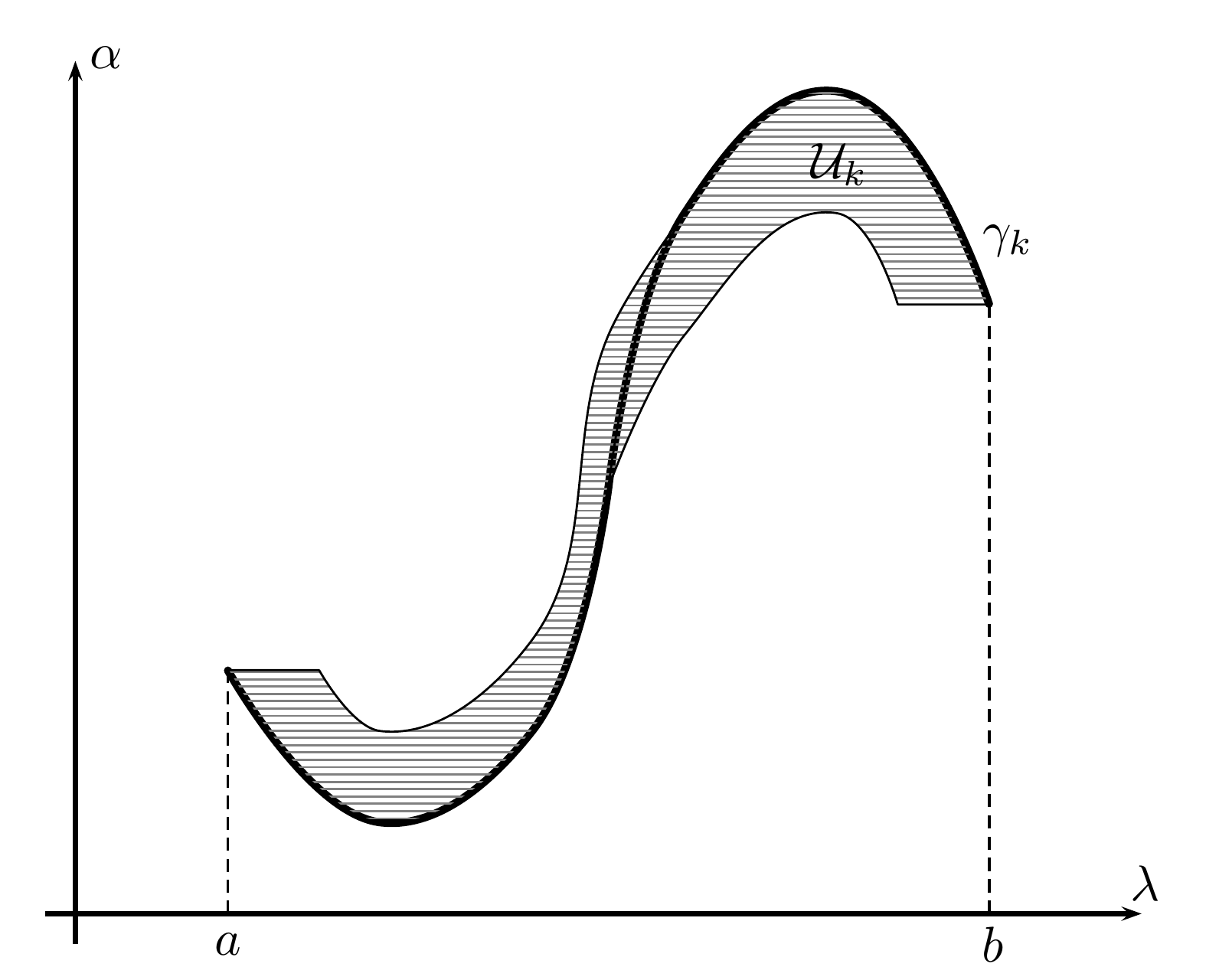}
\caption{}
\label{fig1}
\end{figure}

We observe that in Theorem \ref{a15} we can obtain a similar
result by fixing $\lambda$ in $(a,b)$ and using $\alpha$ as a
parameter. Repeating the proof we get new bifurcation branches of
solutions to \eqref{a1} related to $\lambda$ and $\alpha$ but we
cannot claim that the new set of parameters ${\mathcal U}_k$ is
larger (see Remark \ref{R1}).\\[5pt]

Next we consider a special case where to apply the results of
Theorem \ref{a15}, namely $f(\l, s) =\l e^s$. Here we will get
 more accurate information than those provided in Theorems
\ref{a4}-\ref{a15}.\\
 First of all,  we recall that all radial solutions to \eqref{i2}
 are given by, for $\l\in(0,2)$,
\begin{equation}\label{sol}
v_\l(r) = \log\left(
\frac{8{\delta_\l}}{\l({\delta_\l}+r^2)^2}\right) \,
\hbox{where } \, {\delta_\l}={\delta_\l}^{\pm} = \frac{4-\l \pm
\sqrt{16 - 8\l}}{\l}.
\end{equation}
The solution corresponding to
$\delta_\l^+$ is the {\em minimal one} while that corresponding to
$\delta_\l^-$ has {\em Morse index} $\,1$ and they give rise to radial solutions to \eqref{a1} (see Theorem \ref{b1}). As before we set
$u_{\l,\a}(r)=v_\l(r^\frac{2+\a}2)$ where $v_{\l}$ is the solution in \eqref{sol} corresponding to $\delta_{\l}^{-}$ which has Morse index $1$.\\
Our next result deals with the function $\nu_1(\lambda)$ in \eqref{a5},
that we are able to compute explicitly.
\begin{proposition}\label{p-ni1}
Set $f(\l,s)=\l e^s$ and let $v_{\l}$ be the unique radial
solution of \eqref{i2} with Morse index $1$. Then
\begin{equation}\label{ni-exp}
\nu_1(\l)=\frac{\l-2}2
\end{equation}
and the function which achieves $\nu_1(\l)$ is given by
\begin{equation}\label{b5}
\tilde\psi_{1,\l}(r)=r^{\frac{\sqrt{4-2\l}}2}
\,\frac{2(1-r^4)+(1-r^2)^2\sqrt{4-2\l}}{\l(1-r^2)^2+8r^2}.
\end{equation}
\end{proposition}
The computation of \eqref{ni-exp} will be done using the {\em
generalized Legendre equation} which turns out to play a natural role in
solving the linearized equation to \eqref{i2} at $v_\l$ (this was
pointed out in \cite{B} and \cite{GG}).\\
Using this result we can calculate exactly the Morse index of the
radial solution $u_{\l,\a}$ and also the values of $\l$
where the bifurcation occurs applying Theorems \ref{a11}-\ref{a15} (see Theorems \ref{t1.5}-\ref{texp-bif1-bis}).\\
Let us state the bifurcation result for the case of the exponential nonlinearity, where the constant $\left(\frac{\a+2}2\right)^2$ is merged into the equation, i.e.
\begin{equation}\label{exp}
\begin{cases}
-\Delta u = \mu\abs{x}^{\alpha}e^u, & \hbox{ in }B_1 \\
u > 0, & \hbox{ in }B_1 \\
 u = 0, & \hbox{ on } \partial B_1
\end{cases}
\end{equation}
where $\mu\in\left(0,\frac{(2+\a)^2}2\right)$.
It is worth noting that using some $L^\infty$ estimates proved in
\cite{CL} and \cite{BCLT} we give a more detailed description of
the branches of nonradial solutions. This leads to the following
\begin{theorem}
\label{texp-bif1} Let $\a>0$ be fixed and let $u_{\mu,\a}$ be the radial solution of \eqref{exp} which is not minimal (see \eqref{b2b}).
There are $j$ values
\begin{equation}\label{b8}
\mu_k=\frac{(2+\a)^2}2-2k^2\hbox{ for }k=1,\dots,j,\hbox{ with }
j=\begin{cases}
1+\left[\frac \a2\right] & \hbox{if } \frac \a2\notin \N\\
\frac \a2 &\hbox{ if } \frac \a2\in \N
 \end{cases}
\end{equation}
such that there exists a branch of nonradial solutions of
\eqref{exp} bifurcating from $(\mu_k,u_{\mu_k,\a})$. The
bifurcation is global, the branches are separated and
solutions blow up at $\mu=0$.\\
Furthermore, for any $\mu\in (0,\mu_j)$ problem \eqref{exp} has at
least $j$ nonradial solutions.\\
Finally, for any $\mu\in (\mu_{k+1},\mu_k)$ problem \eqref{exp}
has at least $k$ nonradial solutions.
\end{theorem}
As before let us introduce the curves
\begin{equation}\label{gammak}
\gamma_k=\left\{(\mu,\a)\in\left(0,\frac{(2+\a)^2}2\right)\times(0,+\infty)\hbox{
such that }2\mu=(2+\alpha)^2-4k^2\right\}
\end{equation}
\begin{remark}\label{t-biffin}
As a consequence of Theorem \ref{texp-bif1} we get, for any $k\ge1$ and for any fixed $\a$, the existence of a branch starting from $\gamma_k$ that reaches $\mu = 0$. Then varying $\a$ we recover the region $(\mu,\a)\in\left(0,\frac{(2+\a)^2}2\right)\times(0,+\infty)$ such
that $(\mu,\a)$ lies to the left of $\gamma_k$. In this region there exist at least $k$ nonradial solutions to \eqref{exp} (see figure \ref{fig2}).
\end{remark}

\begin{figure}[t!]
\includegraphics[scale=0.5]{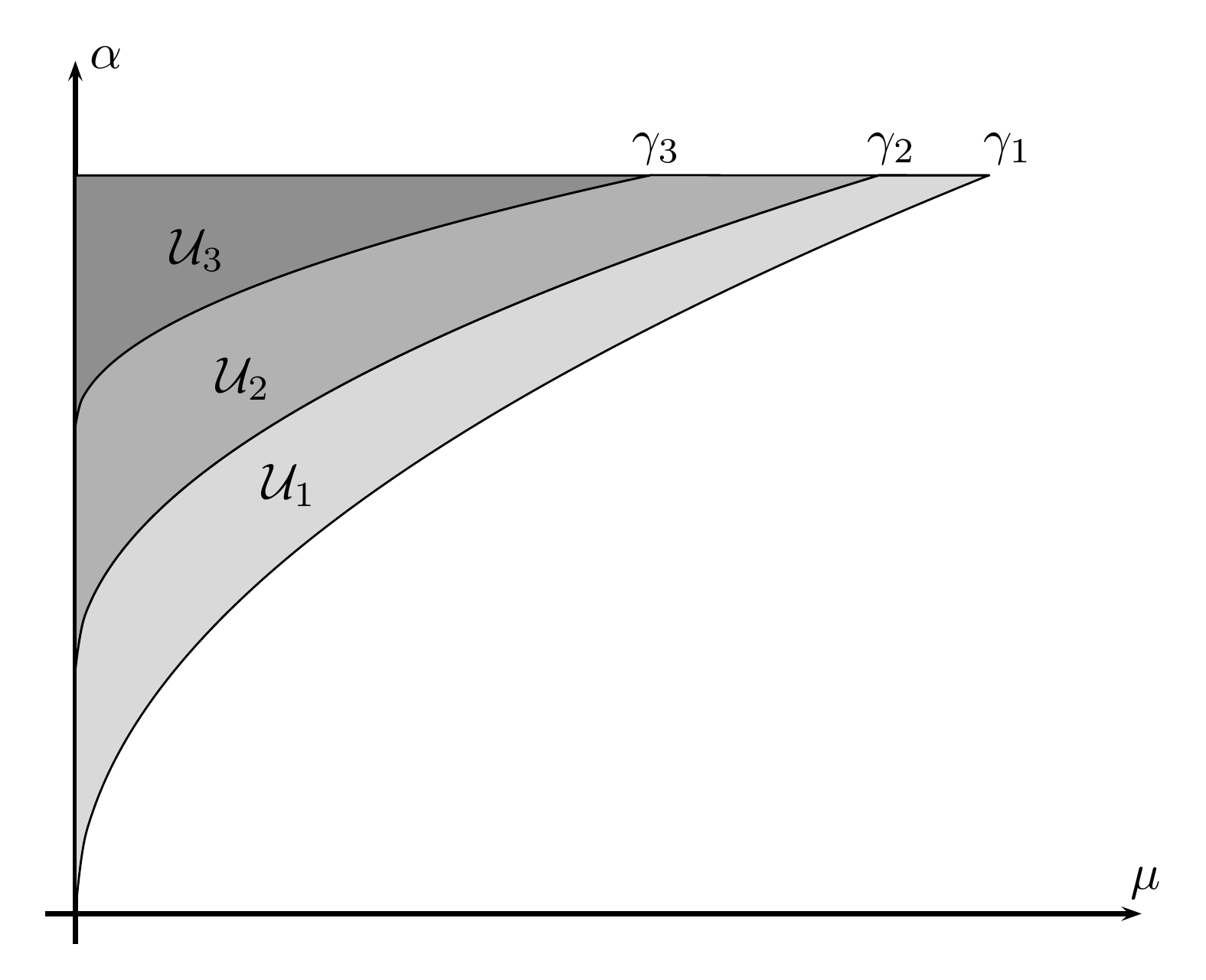}
\caption{$\mathcal{U}_k$ are the pairs to the left of $\gamma_k$}
\label{fig2}
\end{figure}

We end comparing Theorem \ref{t-biffin} with other known existence
results. To do this we need the following estimate, which is, in
our opinion, interesting in itself.
\begin{proposition}\label{stima-pohozaev}
Any smooth solution $u$ of \eqref{exp} must satisfy
\begin{equation}\label{stima-e^u}
2\pi\left(2+\a-\sqrt{(2+\a)^2-2\mu}\right)\leq
\mu\int_{B_1}|x|^{\a}e^u\leq
2\pi\left(2+\a+\sqrt{(2+\a)^2-2\mu}\right).
\end{equation}
\end{proposition}
The bound on $\mu\int_{B_1}|x|^{\a}e^u$  in \eqref{stima-e^u} provides
some interesting information on solutions to \eqref{exp}.
In fact the bounds in \eqref{stima-e^u} correspond to
$\mu\int_{B_1}|x|^{\a}e^u$ when $u$ are the radial solutions to
\eqref{exp}. The first inequality is trivial, being $u$ the
minimal solution. The second one highlights that the other radial
solution ``maximizes" the quantity $\mu\int_{B_1}|x|^{\a}e^u$
(recall that there exist no solution to \eqref{exp} which stays
above any other solution).\\
Moreover, passing to the limit as $\mu\rightarrow0$ in \eqref{stima-e^u} we
get
\begin{equation}\label{r1}
\limsup\limits_{\mu\rightarrow0}\,\mu\int_{B_1}|x|^{\a}e^u\le8\pi\left(1+\frac\alpha2\right).
\end{equation}
 In \cite{DKM} the authors proved the existence of a
solution $u_{\mu,\a}$ concentrating
 as $\mu\rightarrow0$ at $j$ points with $1<j<1+\left[\frac\alpha2\right]$. Each concentration
  point carries a ``quantized mass'' $\mu\int_{B_1}|x|^{\a}e^u=8\pi$.
 Then, by \eqref{r1} we derive that $j=1+\left[\frac\alpha2\right]$ is
 the maximum number of peaks for which such a solution exists. So the result
 in \cite{DKM} in the unit ball is sharp when $\a$ is not an even integer.

\noindent Moreover, the pair $(\mu,\a)$ in \cite{DKM} such that there exist
solutions to \eqref{exp} identify a narrow strip close to $\mu=0$
in the plane $(\mu,\a)$. Our result shows that this region can be
extended till to the curve $\gamma_k$ (see figure \ref{fig3}).
We think that the nonradial solutions in Theorem \ref{texp-bif1} are the same found in \cite{DKM}.

The paper is organized as follows: in Section \ref{appA} we prove some preliminary results which are crucial in the proof of the main theorems. In our opinion they are interesting in themselves. In Section \ref{s3} we consider
the general problem \eqref{a1} and we prove Proposition \ref{a4}
and Theorems \ref{a11}-\ref{a15}. In Section \ref{s4} we turn to
the exponential case and we prove Propositions \ref{p-ni1}, \ref{stima-pohozaev} and Theorem \ref{texp-bif1}. 

Finally in Section \ref{s2}, using the transformation $r\mapsto
r^\frac{2+\a}2$, we retrieve some results, partly known and partly
new, for the problem
\begin{equation*}
\begin{cases}
-\Delta u = \abs{x}^{\alpha}e^u, & \hbox{ in }\R^2 \\
\int_{\R^2} \abs{x}^{\alpha}e^u < +\infty.
\end{cases}
\end{equation*}
while in the Appendix we prove some
some technical results.

\begin{figure}[t!]
\includegraphics[scale=0.5]{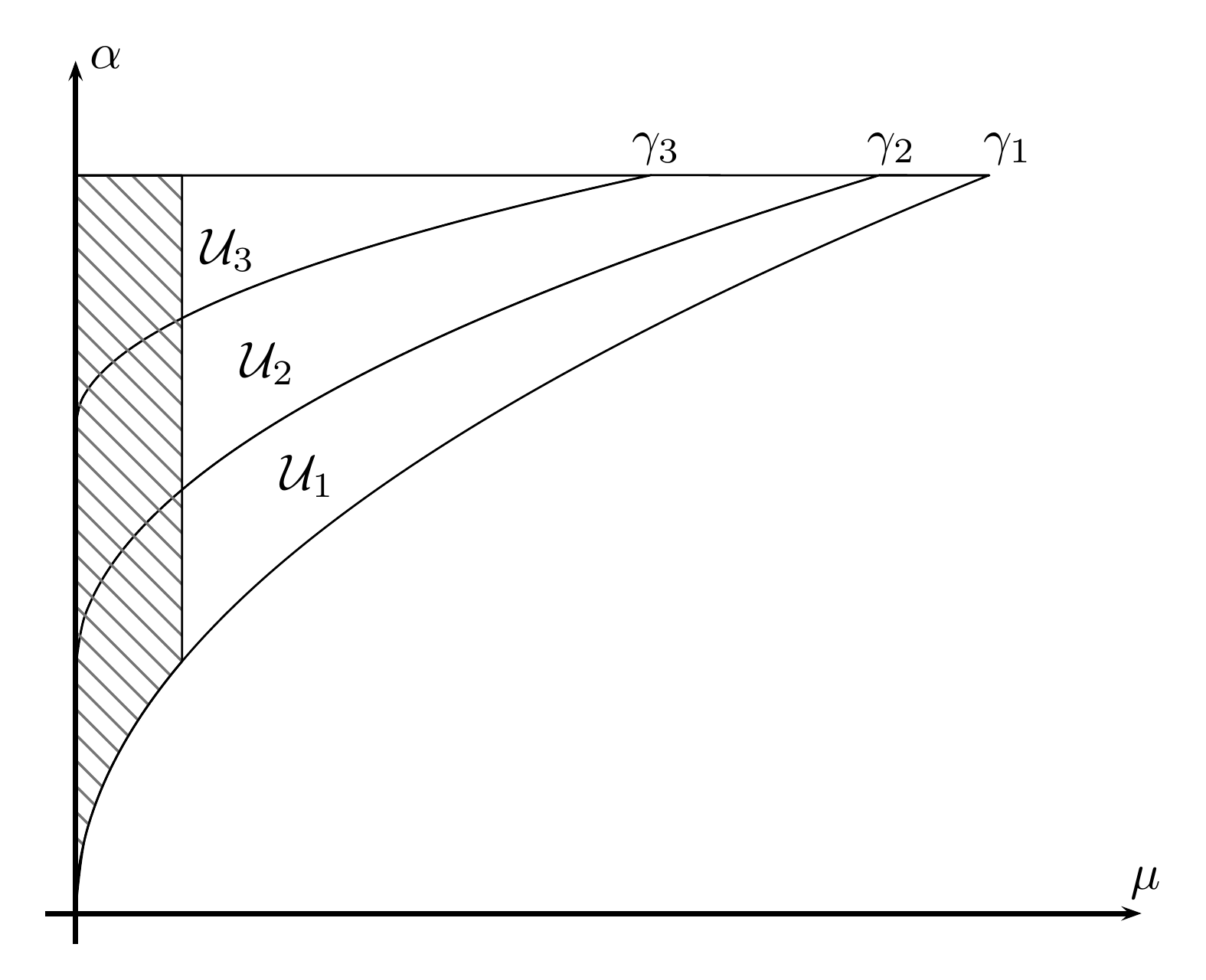}
\caption{}
\label{fig3}
\end{figure}
{\bf Acknowledgment}. The authors wish to thank F. De Marchis, Isabella Ianni and Filomena Pacella for their helpful comments on the definition of (1.6).

\section{Preliminaries}\label{appA}
In this section we prove some useful results which we will apply in the proof of Proposition \ref{a4} and Theorem \ref{a11}. These results link the eigenvalue problem associated to the linearization to \eqref{i2} at the radial solution $u_{\l,\a}$ with a sort of 'eigenvalue' problem with singular weight.\par
The main difficulty is the fact that this linear operator with singular weight is neither compact in $H^1_0(B_1(0))$, so we cannot apply the standard eigenvalue theory, and neither well defined.
To overcame the second problem we define a suitable subspace of $H^1_0(B_1(0))$ in which we can define the quadratic form associated to the linear weighted operator. To this end we let  $\mathcal{H}$ be the closure of $C^{\infty}_0(\Omega)$ with respect to the norm $\nor \eta \nor^2_{\mathcal{H}}=\int_\Omega\left(|\nabla \eta |^2+\frac{\eta^2}{|x|^2}\right)dx$.\par
For what concerns the compactness the problem is harder and indeed we cannot speak of eigenvalues. But the interesting fact is that if the infimum of the quadratic form associated to the linear problem with weight in $\mathcal{H}$ is {\em strictly negative} then there exists in $\mathcal{H}$ a weak solution to the weighted eigenvalue problem related to a negative eigenvalue, see Proposition \ref{B1}, and the same can be done for the subsequent eigenvalues, using the orthogonality with respect the $L^2$ scalar product with weight $\frac 1{|x|^2}$, see Proposition \ref{B111}.\\
Summarizing, each time the infimum of the quadratic form on a suitable orthogonal subspace of $\mathcal{H}$ is strictly negative then we can find a solution to the singular weighted problem. In some sense the negativeness of the eigenvalue restores a sort of compactness. The hypothesis on the negativeness is crucial since we prove in Proposition \ref{P1} that the result is not true if the infimum is zero.

\begin{proposition}\label{B1}
Let $\Omega\subset\R^2$ be a bounded domain with $0\in\Omega$, $a\in L^\infty(\Omega)$ and
\begin{equation}\label{0}
\Lambda_1=\inf_{\eta\in \mathcal{H}\atop \eta\not\equiv0
}\frac{\int_\Omega|\nabla\eta|^2-\int_\Omega a(x)\eta^2}{\int_\Omega \frac{\eta^2}{|x|^2}}<0.
\end{equation}
Then $\Lambda_1$ is achieved.
The function $\psi_1 \in \mathcal{H}$ that achieves $\Lambda_1$ is  strictly positive in $\Omega\setminus\{0\}$ and satisfies
\begin{equation}\label{AAAA}
\int_{\Omega}\nabla \psi_1\cdot \nabla \phi-a(x)\psi_1\phi\, dx=\Lambda_1\int_{\Omega}\frac{\psi_1\phi}{|x|^2}\, dx 
\end{equation}
for any $\phi\in \mathcal{H}$, 
and the eigenvalue $\Lambda_1$ is simple. 
\end{proposition}
\begin{proof}
Let us consider a minimizing sequence $\eta_n\in \mathcal{H}$ for $\Lambda_1$, i.e.,
\begin{equation}\label{1}
\frac{\int_\Omega|\nabla\eta_n|^2-\int_\Omega a(x)\eta_n^2}{\int_\Omega \frac{\eta_n^2}{|x|^2}}=\Lambda_1+o(1).
\end{equation}
Let us normalize $\eta_n$ such that 
\begin{equation}\label{2}
\int_\Om\eta_n^2=1.
\end{equation}
 Then, since $\Lambda_1<0$,  by \eqref{1} we get
\begin{equation}\label{3}
\int_\Omega|\nabla\eta_n|^2-\int_\Omega a(x)\eta_n^2\le0
\end{equation}
and then, since $a$ is bounded and \eqref{2} we deduce from \eqref{3} that
\begin{equation}\label{4}
\int_\Omega|\nabla\eta_n|^2\le C\int_\Omega\eta_n^2\le C.
\end{equation}
Hence $\eta_n\rightharpoonup\eta$ weakly in $H^1_0(\Om)$ and then it holds,
\begin{equation}\label{5a}
\int_\Omega|\nabla\eta|^2\le\liminf_{n\rightarrow+\infty}\int_\Omega|\nabla\eta_n|^2
\end{equation}
\begin{equation}\label{5b}
\int_\Omega a(x)\eta_n^2\rightarrow\int_\Omega a(x)\eta^2.
\end{equation}
So we get
\begin{equation}
\int_\Omega|\nabla\eta|^2-\int_\Omega a(x)\eta^2\le \liminf_{n\rightarrow+\infty}\int_\Omega|\nabla\eta_n|^2-\int_\Omega a(x)\eta_n^2+o(1)
\end{equation}
which implies, since $1=\int_\Omega\eta_n^2\le C\int_\Omega \frac{\eta_n^2}{|x|^2}$,
\begin{equation*}
\frac{\int_\Omega|\nabla\eta|^2-\int_\Omega a(x)\eta^2}{\limsup\limits_{n\rightarrow+\infty}\int_\Omega \frac{\eta_n^2}{|x|^2}}\le\Lambda_1,
\end{equation*}
and elementary properties of $\liminf$ and $\limsup$ imply
\begin{equation}\label{5}
\limsup\limits_{n\rightarrow+\infty}\frac{\int_\Omega|\nabla\eta|^2-\int_\Omega a(x)\eta^2}{\int_\Omega \frac{\eta_n^2}{|x|^2}}\le\Lambda_1,
\end{equation}
Moreover, by Fatou's lemma we have
\begin{equation}\label{6a}
\int_\Omega\frac{\eta^2}{|x|^2}\le\liminf_{n\rightarrow+\infty}\int_\Omega \frac{\eta_n^2}{|x|^2}\le C
\end{equation}
Indeed, if by contradiction we have that $\int_\Omega \frac{\eta_n^2}{|x|^2}\rightarrow+\infty$, by \eqref{4} and \eqref{5b} we derive
\begin{equation}
\lim\limits_{n\rightarrow+\infty}\frac{\int_\Omega|\nabla\eta_n|^2-\int_\Omega a(x)\eta_n^2}{\int_\Omega \frac{\eta_n^2}{|x|^2}}=0
\end{equation}
a contradiction with \eqref{1} because $\Lambda_1<0$.

Hence, again using that $\Lambda_1<0$, we get from \eqref{5} that
\begin{equation}\label{6}
\int_\Omega|\nabla\eta|^2-\int_\Omega a(x)\eta^2<0.
\end{equation}
On the other hand, from \eqref{6a} we get
\begin{equation}\label{7}
\limsup\limits_{n\rightarrow+\infty} \frac 1{\int_\Omega \frac{\eta_n^2}{|x|^2}}=\frac1{\liminf\limits_{n\rightarrow+\infty}\int_\Omega \frac{\eta_n^2}{|x|^2}}\le\frac1{\int_\Omega\frac{\eta^2}{|x|^2}}
\end{equation}
and so by \eqref{6a} and \eqref{6} we get,
\begin{equation}\label{8}
\frac{\int_\Omega|\nabla\eta|^2-\int_\Omega a(x)\eta^2}{\int_\Omega \frac{\eta^2}{|x|^2}}\le\liminf_{n\rightarrow+\infty}\frac{\int_\Omega|\nabla\eta|^2-\int_\Omega a(x)\eta^2}{\int_\Omega \frac{\eta_n^2}{|x|^2}}.
\end{equation}
Finally by \eqref{5} we get
\begin{equation}
\frac{\int_\Omega|\nabla\eta|^2-\int_\Omega a(x)\eta^2}{\int_\Omega \frac{\eta^2}{|x|^2}}\le\Lambda_1
\end{equation}
which proves that $\Lambda_1$ is achieved.\par
Let $\psi_1$ be the function that attains $\Lambda_1$. It is standard to prove that $\psi_1$ solves \eqref{AAAA}. We can assume that $\psi_1\geq 0$ in $\Omega\setminus\{0\}$ otherwise we can substitute $\psi_1$ with $|\psi_1|$. Then we claim that $\psi_1>0$ in  $\Omega\setminus\{0\}$. If this is not the case then we can find  points $x_0,x_1\in\Omega\setminus\{0\}$ such that $\psi_1(x_0)=0$ and $\psi_1(x_1)>0$. We choose a regular open subset $\Omega'\subset \Omega\setminus B_{\e}(0)$  (with $\e>0$ small enough) that contains $x_0$ and  $x_1$.
From \eqref{AAAA}, since $0$ does not belong to $\Omega'$, the function $\psi_1$ satisfies
$$\begin{cases}
-\Delta \psi_1-a(x)\psi_1-\frac {\Lambda_1}{|x|^2}\psi_1=0 & \hbox{ in }\Omega'\\
\psi_1\geq 0 & \hbox{ in } \overline{\Omega}'
\end{cases}
$$
Then the strong maximum principle implies that $\psi_1>0$ in $\Omega'$ or $\psi_1\equiv 0$. The first case contradicts the fact that $\psi_1(x_0)=0$ while the second contradicts $ \psi_1(x_1)>0$ and this contradiction proves the claim.\\
Now we prove that $\Lambda _1$ is simple. Assume, by contradiction that there exists another function $\psi_2$ that attains $\Lambda_1$. We can assume that $\psi_2$ is orthogonal to $\psi_1$, in the sense that satisfies
\begin{equation}\label{BBBB}
\int_{\Omega}\frac{\psi_1\psi_2}{|x|^2}\, dx=0.
\end{equation}
If this is not the case we let $\widetilde \psi_2=\psi_2-c\psi_1$ whit $c=\frac{\int_{\Omega}\frac{\psi_1\psi_2}{|x|^2}\, dx}{\int_{\Omega}\frac{\psi_1^2}{|x|^2}\, dx}$ so that \eqref{BBBB} is satisfied. From \eqref{BBBB} we derive that $\psi_2$ changes sign in $\Omega\setminus\{0\}$. From \eqref{0} we have that also $|\psi_2|$  attains $\Lambda_1$
and, since $|\psi_2|\geq 0$ in $\Omega\setminus\{0\}$
reasoning as before we get that $|\psi_2|> 0$ in $\Omega\setminus\{0\}$, contradicting the fact that $\psi_2$ changes sign. This proves that $\Lambda_1$ is simple and concludes the proof.
\end{proof}
Next proposition shows  that the condition \eqref{5} is sharp.
\begin{proposition}\label{P1}
Let us assume that in the previous proposition $a\equiv0$. Then
\begin{equation*}
\bar\Lambda_1=\inf_{\eta\in  \mathcal{H} \atop \eta\not\equiv0
}\frac{\int_{B_1}|\nabla\eta|^2}{\int_{B_1} \frac{\eta^2}{|x|^2}}=0
\end{equation*}
and it is not achieved.
\end{proposition}
\begin{proof}
Choosing the test functions
\begin{equation*}
\eta_\e(x)=
\begin{cases}
1-|x|&\hbox{if }\e\le|x|\le1\\
\frac{2(1-\e)}\e|x|+\e-1&\hbox{if }\frac\e2\le|x|\le\e\\
0&\hbox{if }|x|\le\frac\e2
\end{cases}
\end{equation*}
A straightforward computation shows that 
\begin{equation}\label{10}
\frac{\int_{B_1}|\nabla\eta_\e|^2}{\int_{B_1} \frac{\eta_\e^2}{|x|^2}}=\frac{\frac 7 4\pi+o(1)}{2\pi (-1+o(1))\log\e}\longrightarrow0\quad\hbox{as }\e\rightarrow0,
\end{equation}
which proves \eqref{9}. Of course $\bar\Lambda_1$ is not achieved.
\end{proof}
Now we consider the other 'eigenvalues' with weight. We are able to prove that, if they are negative then they are attained and an 'eigenfunction' that belongs to $\mathcal{H}$ exists. In sequel we say that $\eta$ and $\psi$ in $\mathcal{H}$ are orthogonal if they satisfy $\int_{\Omega}\frac{\eta\psi}{|x|^2}\, dx=0.$ 
\begin{proposition}\label{B111}
Let us assume $\Om$, $\Lambda_1$, $\psi_1$ and $a(x)$ as in the proposition \ref{B1}. 
Then if we have that
\begin{equation}\label{9}
\Lambda_2=\inf_{\scriptstyle
\eta\in\mathcal{H} ,\atop\scriptstyle
\eta\perp\psi_1}\frac{\int_\Omega|\nabla\eta|^2-\int_\Omega a(x)\eta^2}{\int_\Omega \frac{\eta^2}{|x|^2}}<0
\end{equation}
then $\Lambda_2$ is achieved. Moreover the functions $\psi_2\in \mathcal{H}$ that attains $\Lambda_2$ satisfies
$$
\int_{\Omega}\nabla \psi_2\cdot \nabla 
\phi-a(x)\psi_2\phi\, dx=
\Lambda_2\int_{\Omega}\frac{\psi_2\phi}{|x|^2}\, dx $$
for any $\phi\in \mathcal{H}$.\par
Similarly for $i=2,..,k$, if we have that
\begin{equation}
\Lambda_i=\inf_{\scriptstyle
\eta\in\mathcal{H} ,\atop\scriptstyle
\eta\perp span\left\{\psi_1,\psi_2,..,\psi_{i-1}\right\}}\frac{\int_\Omega|\nabla\eta|^2-\int_\Omega a(x)\eta^2}{\int_\Omega \frac{\eta^2}{|x|^2}}<0
\end{equation}
then $\Lambda_i$ is achieved and the functions $\psi_i\in \mathcal{H}$ that attain $\Lambda_i$ satisfy
$$
\int_{\Omega}\nabla \psi_i\cdot \nabla 
\phi-a(x)\psi_i\phi\, dx=
\Lambda_i\int_{\Omega}\frac{\psi_i\phi}{|x|^2}\, dx $$
for any $\phi\in \mathcal{H}$.
\end{proposition}
\begin{proof}
It is the same of Proposition \ref{B1}. For any $i$ let us consider a minimizing sequence $\eta_{i,n}\in \mathcal{H} $ for $\Lambda_i$. Then it converges to a function $\eta_i$ which achieves $\Lambda_i$.
\end{proof}
\begin{lemma}\label{B2}
Let us consider a solution to
\begin{equation}\label{f1a}
\begin{cases}
-\psi'' - \frac{1}{r}\psi' +\beta^2\frac{\psi}{r^2}=h\psi, \quad \text{in} \ \ (0,1)\\
\psi(1)=0,\ \int_0^1\left((\psi')^2+\frac{\psi^2}{r^2}\right)rdr< \infty
\end{cases}
\end{equation}
with $h\in L^\infty(0,1)$ and $\beta\ne0$. Then $\psi\in L^\infty(0,1)$ and $\psi(0)=0$.
\end{lemma}
\begin{proof}
Since $\int_0^1\frac{\psi^2}rdr<+\infty$ then there exists a sequence $r_n\rightarrow0$ such that
$\psi(r_n)=o(1)$ as $n\rightarrow+\infty$. Such a sequence exists because, if not, we get $\psi(r)\ge C$ in a suitable neighborhood of $0$ and this contradicts that $\int_0^1\frac{\psi^2}rdr<+\infty$.\\
Let us observe that the function $v(r)=r^{\beta}$, with $\beta > 0$ satisfies
 \begin{equation}\label{z2}
\begin{cases}
-v''-\frac1r v'+\frac{\beta^2}{r^2}v=0, \quad\hbox{in }(0,+\infty)\\
v(0)=0.
\end{cases}
\end{equation}
 From \eqref{z2} and \eqref{f1a} we obtain, integrating on $(r_n,R)$,
 \begin{equation}\label{f2}
 \int_{r_n}^Rs^{\beta+1}h(s)\psi(s)ds=-R^{\beta+1}\psi'(R)+r_n^{\beta+1}\psi'(r_n)+\beta R^{\beta}\psi(R)-\beta r_n^{\beta}\psi(r_n)
\end{equation}
We claim that
\begin{equation}\label{f3}
r_n^{\beta+1}\psi'(r_n)=o(1)
\end{equation}
Integrating \eqref{f1a} we get
\begin{equation*}
r_n^{\beta+1}\psi'(r_n)=O\left(r_n^{\beta}\right)+r_n^{\beta}
\int_{r_n}^1 sh(s)\psi(s)ds
-r_n^{\beta}\int_{r_n}^1\frac{\beta^2}s\psi(s)ds=o(1)
\end{equation*}
since $r_n^{\beta}\int_{r_n}^1\frac{\psi(s)}s\le r_n^{\beta} \left(\int_{r_n}^1\frac{\psi^2(s)}s\right)^\frac12
\left(\int_{r_n}^1\frac1s\right)^\frac12\le C r_n^{\beta} \left(-\log r_n\right)^\frac12=o(1)$
which proves  \eqref{f3}. Hence  \eqref{f2} becomes
 \begin{equation}\label{f4}
 \int_0^R s^{\beta +1}h(s)\psi(s)ds=-R^{\beta+1}\psi'(R)+\beta R^{\beta}\psi(R)
\end{equation}
Then we deduce that
 \begin{equation}\label{f5}
\frac{\psi(t)}{t^{\beta}}=\int_t^1\frac1{R^{2\beta+1}}
\left(\int_0^Rs^{\beta+1}h(s)\psi(s)ds\right)dR.
\end{equation}
Now, since $h\in L^\infty(0,1)$ we get
\begin{align}\label{B.12}
\left|\int_0^Rs^{\beta+1}h(s)\psi(s)\right|\le C\int_0^Rs^{\beta+\frac32}\frac{\psi(s)}{s^\frac12} & \le C \left(\int_0^Rs^{2\beta+3}\right)^\frac12
\left(\int_0^R\frac{\psi^2(s)}s\right)^\frac12\nonumber \\
&\leq CR^{\beta+2}.
\end{align}
Finally \eqref{f5} becomes
\begin{equation}\label{f6}
\psi(t)=
O(t^\beta)
\end{equation}
and this shows that $\psi(0)=0$. This ends the proof.
\end{proof}

\begin{lemma}\label{B3}
Let $v_\l$ be a solution of \eqref{i2} with Morse index $1$. Then there exists only one negative number $\nu_1(\l)$ such that the problem 
\begin{equation}\label{fin-fin}
\begin{cases}
-\psi'' - \frac{1}{r}\psi'-f'(\l,v_\l)\psi= \frac{ \nu_1(\l) }{r^2}\psi , \quad \text{in} \ \ (0,1)\\
\psi(1)=0,\ \int_0^1 r(\psi')^2+\frac{\psi^2}r\, dr< \infty
\end{cases}
\end{equation}
 has solution. Moreover, $\nu_1$ is simple and the function that satisfies \eqref{fin-fin} is strictly positive.
\end{lemma}
\begin{proof}
By assumption $v_\l$ is a Morse index one solution to \eqref{i2}. This implies that the eigenvalue problem
\begin{equation}\label{eig}
\begin{cases}
-\Delta w-f'(\l,v_\l)w=\l w &\hbox{ in }B_1(0)\\
w=0&\hbox{ on }\partial B_1(0)
\end{cases}
\end{equation}
has only one negative eigenvalue $\l_1$ in $H^1_0(B_1(0))$. It is well known that the first eigenfunction $w_1$ is radial. 
Since the first eigenvalue $\l_1$ depends with continuity from the domain $\Omega$, see \cite{BNV}, we have that for $\e>0$ sufficiently small, in the domain $A_\e=B_1(0)\setminus B_{\e}(0)$, the eigenvalue problem  \eqref{eig} in $H^1_0(A_\e)$ has at least $1$ negative eigenvalues $\l_1(A_\e)$. Moreover the strict monotonicity of the first eigenvalues with respect to the domain, see \cite{BNV}, implies  that $\l_{2}(A_\e)> \l_{2}(B_1(0))\geq 0$.\\
In $A_\e$ we can consider the weighted eigenvalue problem
\begin{equation}\label{eig2}
\begin{cases}
-\Delta w-f'(\l,v_\l)w=\Lambda \frac {w}{|x|^2} &\hbox{ in }A_\e\\
w=0&\hbox{ on }\partial A_\e
\end{cases}
\end{equation}
Reasoning exactly as in Lemma 2.1 of \cite{GGPS}, we have that the two eigenvalue problems are equivalent in $A_\e$. Then problem  \eqref{eig2} has a first negative eigenvalue $\Lambda_1(A_\e)$ which is attained on a radial function $\psi_\e>0$. Extending $\psi_\e$ to zero in $B_\e(0)$ we have that $\psi_\e$ is a radial function that belongs to $\mathcal{H}$. Then $\psi_\e$ belongs to the space $H$ defined in Proposition \ref{a4}. Using $\psi_\e$ in the definition of $\nu_1$ in \eqref{a5} then we get that $\nu_1<0$.\\
Then we can repeat exactly the proof of Proposition \ref{B1} (in this radial case) getting that there exists a radial strictly positive function $\tilde{\psi}_1$ that attains $\nu_1$ and hence satisfies \eqref{fin-fin}.\\
Let us suppose that there exists $\nu_2<0$, $\nu_2\ne\nu_1$ and $\psi_2$ satisfying
\begin{equation} \label{+++}
\begin{cases}
-\psi_2'' - \frac{1}{r}\psi_2' - f'(\l,v_\l)\psi_2 = \frac{\nu_2}{r^2}\psi_2, \quad \text{in} \ \ (0,1)\\
\psi_2(1)=0,\ \int_0^1\left(r(\psi_2')^2+ \frac{\psi^2}r\right)dr< \infty,
\end{cases}
\end{equation}
We claim that there exists a sequence $r_n\to 0$ such that $r_n\psi_2(r_n)\psi_1'(r_n)=o(1)$. If not we get $\psi_2(r)\psi_1'(r)\geq \frac Cr$ in a suitable 
neighborhood of $0$ and this contradicts that $\int_0^1 \psi'_1(r)\psi_2(r)\, dr\leq \left(\int_0^1 r(\psi_1')^2\,dr \right)^{\frac 12}\left( \int_0^1 \frac {\psi_2^2}r\,dr\right)^{\frac 12}\leq C$. Then we can multiply equation \eqref{fin-fin} by $\psi_2$, integrate over $(r_n,1)$ and pass to the limit as $n\to +\infty$ getting
$$\int_0^1 r\psi_1'\psi_2'\, dr -\int_0^1 r f'(\l,v_\l)\psi_1\psi_2\, dr =\nu_1(\l)\int_0^1  \frac {\psi_1\psi_2}r\, dr.$$
In the same way, using equation \eqref{+++} we get
$$\int_0^1 r\psi_2'\psi_1'\, dr -\int_0^1 r f'(\l,v_\l)\psi_2\psi_1\, dr =\nu_2\int_0^1  \frac {\psi_2\psi_1}r\, dr.$$
Then subtracting the equations we 
 have that
\begin{equation}\label{q1}
\int_0^1\frac{\psi_1\psi_2}r=0
\end{equation}
and
\begin{equation}\label{q2}
\int_{B_1}\left[|\nabla \psi_i|^2 - f'(\lambda,v_\lambda)\psi_i^2\right]<0\quad i=1,2.
\end{equation}
By \eqref{q2} we get that both $\psi_1$ and $\psi_2$ make negative the quadratic form associated to $v_\l$. Hence, since the Morse index of $v_\l$ is one, we get that $\psi_1=\alpha\psi_2$ for some $\alpha\in\R$. But this contradicts \eqref{q1}.
\end{proof}

\begin{lemma}\label{B100}
Let $u_{\l,\a}$ be a solution to \eqref{a1} whose Morse index is $M>0$. Then there exist exactly $M$ functions $\psi_i$ and $M$ numbers $\Lambda_i<0$ such that the problem
\begin{equation}\label{fin-fine}
\begin{cases}
-\Delta \psi-\left(\frac{2+\a}2\right)^2|x|^{\a}f'(\l,u_{\l,\a})\psi= \frac{ \Lambda }{|x|^2}\psi , \ \text{in} \ \ B_1(0)\setminus\{0\}\\
\psi=0 \ \text{on} \ \partial B_1(0)\\
\int_{B_1(0)}\left(|\nabla \psi|^2+  \frac{\psi^2}{|x|^2}\right)dx< \infty
\end{cases}
\end{equation}
admits a solution. The functions $\psi_i$ can be taken in such a way they verify
\begin{equation}\label{perpendicolari}
\int_{B_1(0)}\frac{ \psi_i\psi_j}{|x|^2}\, dx=0 \quad \text{ for }i\neq j.
\end{equation}
\end{lemma}
\begin{proof}
Since the Morse index of $u_{\l,\a}$ is $M$, then there exist exactly $M$ eigenfunctions $\widehat\psi_1,\dots,\widehat\psi_M\in H^1_0(B_1(0))$, orthogonal in $L^2(B_1(0))$ that satisfy
\begin{equation}\label{fin-fin-fin}
\begin{cases}
-\Delta \widehat\psi_i-\left(\frac{2+\a}2\right)^2|x|^{\a}f'(\l,u_{\l,\a})\widehat\psi_i = \lambda_i \widehat\psi_i  \ \hbox{in} \ \ B_1(0)\\
\widehat\psi_i =0 \ \hbox{on} \ \partial B_1(0).
\end{cases}
\end{equation}
Since the eigenvalues $\l_i$ depend with continuity from the domain $\Omega$, see \cite{MM}, we have that for $\e>0$ sufficiently small, in the domain $A_\e=B_1(0)\setminus B_{\e}(0)$, the eigenvalue problem \eqref{fin-fin-fin} has at least $M$ negative eigenvalues $\l_i(A_\e)$. Moreover the monotonicity of the eigenvalues with respect to the domain implies  that $\l_{M+1}(A_\e)\geq \l_{M+1}(B_1(0))\geq 0$.\\
Summarizing for $\e$ small enough the eigenvalue problem \eqref{fin-fin-fin} in $A_\e$ has exactly $M$ negative eigenvalues.\\
In the domain $A_\e$ the eigenvalue problem \eqref{fin-fin-fin} and the weighted eigenvalue problem \eqref{fin-fine} are equivalent. Then in $A_\e$ the eigenvalue problem \eqref{fin-fine} has exactly $M$ negative eigenvalues $\L_i(A_\e)$ and $M$ negative eigenfunctions $\psi_{i,\e}$ that satisfy
$$
\begin{cases}
-\Delta \psi_{i,\e}-\left(\frac{2+\a}2\right)^2|x|^{\a}f'(\l,u_{\l,\a})\psi_{i,\e} = \frac{\Lambda_i(A_\e)}{|x|^2} \psi_{i,\e}  \ \hbox{in} \ A_\e\\
\psi_{i,\e} =0 \ \hbox{on} \ \partial A_\e.
\end{cases}$$
Since every function in $H^1_0(A_\e)$ extended to zero in $B_\e(0)$ belongs to $\mathcal{H} $ (as defined in Proposition \ref{B1}) then we have, using Propositions \ref{B1} and \ref{B111} that there exist $M$ function $\psi_i$ in $\mathcal{H} $ and $M$ numbers $\L_i<0$ that satisfy \eqref{fin-fine}. \\
If $\psi_i$ and $\psi_j$ satisfy \eqref{fin-fine} with numbers $\L_i\neq\L_j$ then, we can use the equation \eqref{fin-fine} to get 
$$(\L_i-\L_j)\int_{B_1(0)}\frac{\psi_i\psi_j}{|x|^2}\, dx=0$$
so that \eqref{perpendicolari} holds if $\L_i\neq \L_j$. If instead $\psi_i$ and $\psi_j$ are two functions linearly independent that satisfy \eqref{fin-fine} with the same number $\L_i$ then we can construct a new function $\widetilde \psi_j=\psi_j-c\psi_i$ such that $\widetilde \psi_j$ satisfies \eqref{fin-fine} with $\L_i$ (since the equation is linear) and $\psi_i$ and $\widetilde \psi_j$ are orthogonal, in the sense that they  satisfy \eqref{perpendicolari}.\\
Assume, by contradiction that there exists another function $\psi_{M+1}\in \mathcal{H}$ that satisfies \eqref{fin-fine} with a negative number $\L_{M+1}$. Then, as before, we can assume that $\psi_{M+1}$ is orthogonal to $\psi_1,\dots, \psi_M$ in the sense of \eqref{perpendicolari}. Then, using equation \eqref{fin-fine} for the functions $\psi_i$, $i=1,\dots,M+1$ we have that the quadratic form
$$Q(\psi, \psi)=\int_{B_1(0)}|\nabla \psi|^2-\left(\frac{2+\a}2\right)^2|x|^{\a}f'(\l,u_{\l,\a})\psi^2\, dx$$
is negative definite on the $(M+1)$-dimensional subspace of $H^1_0(B_1(0))$ spanned by $\psi_1,\dots,\psi_{M+1}$, contradicting the definition of Morse index of $u_{\l,\a}$.
\end{proof}

\section{The abstract existence result}\label{s3}
\noindent This section is devoted to study the general problem \eqref{a1}. 
\noindent Our first result studies the nondegeneracy of solutions to
\eqref{a1}.

\begin{proof}[Proof of Proposition \ref{a4}: ]
Let $v_{\l}$ be a solution of \eqref{i2}. From the symmetry
results of Gidas, Ni, Nirenberg \cite{GNN} we get that $v_{\l}$ is
radial. Setting $u_{\l,\a}(r)=v_{\l}\left(r^\frac
{2+\a}2\right)$, where $r=|x|$, a straightforward computation
shows that $u_{\l,\a}$ is a radial solution of \eqref{a1}
for any $\lambda\in(a,b)$ and $\alpha>0$.\\
Let us consider the linearized operator of \eqref{a1} at the
radial solution $u_{\l,\a}$, i.e. the problem \eqref{a7}.
Decomposing \eqref{a7} using the
spherical harmonic functions $Y_k(\theta)$, we get that $$w(r,\theta)=\sum_{k=0}^{+\infty}w_k(r)Y_k(\theta)$$ is a solution of
\eqref{a7} if and only if
$w_k(r):=\int_{S^1}w(r,\theta)Y_k(\theta)\, d\theta$ is a solution
of
 \begin{equation}\label{a8}
\begin{cases}
-w_k''-\frac 1r w'_k+\frac {k^2}{r^2}w_k = \left(\frac{2+\alpha}{2}\right)^2r^{\alpha}f'(\lambda ,u_{\l,\a})w_k, \quad \hbox{ in }(0,1) \\
w_k'(0) = 0=w_k(1) \hbox{ if } k=0 , \hbox{ and } w_k(0) =
0=w_k(1) \hbox{ if } k\geq 1\\
\int_0^1 r (w_k')^2+ \frac{w_k^2}r \, dr<\infty.
\end{cases}
\end{equation}
Note that, in order to $w$ be a smooth solution to \eqref{a7}, we must have $w_0'(0)=0$ and $w_k(0)=0$ for $k \ge 1$.
Letting $\eta_k(r)=w_k(r^{\frac 2{2+\a}})$, we have that
$\eta_k$ solves
 \begin{equation}\label{a9}
\begin{cases}
-\eta_k''-\frac 1r \eta'_k - f'(\lambda ,v_{\l})\eta_k= -\frac {4k^2}{(2+\a)^2r^2}\eta_k, \quad \hbox{ in }(0,1) \\
\eta_k'(0) = 0=\eta_k(1) \hbox{ if } k=0 , \hbox{ and } \eta_k(0)
= 0=\eta_k(1) \hbox{ if } k\geq 1 \\
\int_0^1 r (\eta_k')^2 + \frac{\eta_k^2}r \, dr<\infty.
\end{cases}
\end{equation}
First let us consider the case where $k\ge1$. Since the infimum $\nu_1(\l)$ in \eqref{a5} is achieved (see Lemma \ref{B3}), we have that there exists a solution to the problem
 \begin{equation}\label{a10}
\begin{cases}
-\eta''-\frac 1r \eta' - f'(\lambda ,v_{\l})\eta=\nu_1(\l)\eta, \quad \hbox{ in }(0,1) \\
\eta(1) = 0 \\
\int_0^1 r (\eta')^2+ \frac{\eta^2}r \, dr<\infty.
\end{cases}
\end{equation}
Moreover, by Lemma \ref{B2} we have that the solution to  \eqref{a10} satisfies $\eta(0)=0$. 
Finally by Lemma \ref{B3} there exist only one negative number $\nu_1(\l)$  and one positive function $\tilde\psi_{1,\l}$ verifying  \eqref{a10}. This proves  \eqref{a6}.
Using the reversed map
$r\mapsto r^\frac{2+\a}2$ we get that
$\psi_{\l,\a}(r)=\tilde\psi_{1,\l}\left(r^\frac{2+\a}2\right)$
is a solution to \eqref{a8} so that \eqref{a6a} holds.\\
If, else $v_{\hat\l}$ is degenerate then, from \cite{LiNi}, the
linearized operator at $v_{\hat\l}$ has a unique solution $\tilde
\psi_{\hat\l}$ which is radial. Then, letting
$\psi_{\hat\l,\a}(r)=\tilde \psi_{\hat\l}(r^{\frac{2+\a}2})$, we
get that $\psi_{\hat\l,\a}$ is a radial solution of the linearized
operator \eqref{a7} for any $\a>0$.
\end{proof}
Next we compute the Morse index of the radial solution to
\eqref{a1} which proves Theorem \ref{a11}.

\begin{proof}[Proof of Theorem \ref{a11}: ]

The Morse index $M$ of $u_{\l,\a}$ is the number of negative eigenvalues (counted with their multiplicity) of the problem
\begin{equation}\label{a18}
\begin{cases}
-\Delta w -\left(\frac{2+\alpha}{2}\right)^2|x|^{\alpha}f'(\lambda ,u_{\l,\a})w=\L w  & \hbox{ in }B_1 \\
w = 0, & \hbox{ on } \partial B_1\\
\int_{B_1}|\nabla w|^2\, dx<\infty.
\end{cases}
\end{equation}
Observe that, since $\left(\frac{2+\alpha}{2}\right)^2|x|^{\alpha}f'(\lambda ,u_{\l,\a})\in L^{\infty}(B_1(0))$, then the Morse index of $u_{\l,\a}$ is finite for any $\a\geq 0$.
Using Lemma \ref{B100}  then we get that
the Morse index of the radial solution $u_{\l,\a}$
coincides with the number of negative eigenvalues $\L$ of the problem, counted with their multiplicity,
\begin{equation}\label{a12}
\begin{cases}
-\Delta \psi-\left(\frac{2+\alpha}{2}\right)^2|x|^{\alpha}f'(\lambda ,u_{\l,\a})\psi=\frac \L{|x|^2}\psi  & \hbox{ in }B_1\setminus\{0\} \\
\psi= 0, & \hbox{ on } \partial B_1\\
\int_{B_1(0)}|\nabla \psi|^2+\frac{\psi^2}{|x|^2}\, dx <\infty
\end{cases}
\end{equation}
Arguing as in the previous proposition, setting $\psi_{i,k}(r):=\int_{S^1}\psi_i(r,\theta)Y_k(\theta)\, d\theta$ and then $\eta_{i,k}(r)=\psi_{i,k}(r^{\frac 2{2+\a}})$, we get that $\eta_{i,k}$ satisfies
\begin{equation}\label{a13}
\begin{cases}
-\eta_{i,k}''-\frac 1r \eta_{i,k}' - f'(\lambda ,v_{\l})\eta_{i,k}=4\frac{\L_i-k^2 }{(2+\a)^2}\frac{\eta_{i,k}}{r^2}, \quad \hbox{ in }(0,1) \\
\eta_{i,k}'(0) = 0=\eta_{i,k}(1) \hbox{ if } k=0 , \hbox{ and } \eta_{i,k}(0)
= 0=\eta_{i,k}(1) \hbox{ if } k\geq 1\\
\int_0^1r(\eta_{i,k}')^2+\frac{\eta_{i,k}}{r}\, dr<\infty
\end{cases}
\end{equation}
for some value of $k$ and $\L_i<0$. Since problem \eqref{a13} admits only one negative eigenvalue $\nu_1(\lambda)$ (see Lemma \ref{B3}), arguing as before, from \eqref{a13} we derive that
\begin{equation}\label{a14}
\nu_1(\lambda)=4\frac{\L_i-k^2 }{(2+\a)^2}.
\end{equation}
So we have that the modes $k$ which contribute to the Morse index
of the solution $u_{\l,\a}$ verify
$\frac{4k^2}{(2+\a)^2}+\nu_1(\l)<0$, i.e. $k<\frac{2+\a}2\sqrt
{-\nu_1(\l)}$. Finally recalling that the dimension of the
eigenspace of the Laplace-Beltrami operator on $S^1$ is $2$ for
any $k\ge1$ then
\eqref{morse-general} follows.\\
\end{proof}
We end this section proving the bifurcation result from the radial solution $u_{\l,\a}$.
\begin{proof}[Proof of Theorem  \ref{a15}: ]
Let $\a$ be fixed and consider the operator $T(\l,v):(a,b)\times C^{1,\g}_0(\bar B_1)\to  C^{1,\g}_0(\bar B_1)$ defined by
$$T(\l,v):=\left(-\Delta\right)^{-1}\left(\left(\frac{2+\alpha}{2}\right)^2|x|^{\alpha}f(\lambda ,v)\right).$$
$T$ is a compact operator for every fixed $\l$ and it is continuous with respect to $\l$. Let us define $S(\l,v):(a,b)\times C^{1,\g}_0(\bar B_1)\to  C^{1,\g}_0(\bar B_1)$  as
$$S(\l,v):=v-T(\l,v).$$
A function $v\in C^{1,\g}_0(\bar B_1)$ is a solution of \eqref{a1}
related to $\l$ if and only if $(\l,v)$ is in the kernel of $S$ and $v>0$ in $B_1$.\\
Using assumption ii) and \eqref{a6} we can find $\e>0$ such that the interval
$(\l(k)-\e,\l(k)+\e)$  does not contain degeneracy points of
\eqref{a1} other than $ \l(k)$.
Moreover the Morse index of the radial solution $u_{\l,\a}$ changes at $\l(k)$
because $\nu_1(\lambda)+\frac{4k^2}{(2+\alpha)^2}$ changes sign at $ \l(k)$.
Since the eigenspace of the Laplace Beltrami operator on $S^1$ related to $k$
is spanned by $\{\cos k\theta,\sin k\theta\}$, 
we have that, for $k\ge1$, the change in the Morse index is
exactly 2.
However, if we restrict to the space
\begin{equation}\label{X}
X=\left\{v\in C^{1,\g}_0(\bar B_1)\hbox{ such that }v=v(r,\theta)\hbox{ and }v\hbox{ is even in }\theta\right\}
\end{equation}
then only the spherical harmonic $\cos k\theta$ contributes to the Morse index and the change in the Morse index of $u_{\l,\a}$ in $X$ at the point $\l(k)$ is exactly one.\\
To prove the bifurcation result we use the cones
\begin{equation*}
\mathcal{C}^k:=\left\{
\begin{split}
v\in X, \hbox{ such that }v\geq 0 \hbox{ in }B_1\,,\,v(r,\theta)=v\left(r,\theta+\frac{2\pi}k\right), \hbox{ for }r\in[0,1],\\
\theta\in[0,2\pi], v(r,\theta) \hbox{ non increasing in }\theta \hbox{ for }0\leq\theta\leq \frac{\pi}k\, ,\, \,0\leq r\leq 1
\end{split}\right\}
\end{equation*}
introduced by Dancer in \cite{DA}.

The operator $S(\l,v)$ then maps $(a,b)\times \mathcal{C}^k\to \mathcal{C}^k$ (see \cite[Lemma 1]{DA}) and the compactness of $T$ allows us to compute the Leray-Schauder degree of $S$ in a suitable neighborhood of the radial solution $(\l(k), u_{\l(k),\a})$. The odd change in the Morse index of $u_{\l(k),\a}$ in $ \l(k)$ causes a change in the Leray-Schauder degree along the curve of radial solutions $(\l,u_{\l,\a})$ going from $(\l(k)-\e,u_{\l(k)-\e,\a})$ to $( \l(k)+\e,u_{\l(k)+\e,\a})$ so that the bifurcation occurs. Finally, since it is not difficult to see that $\mathcal{C}^k\cap \mathcal{C}^h$ contains only radial solutions (\cite{DA83}), then branches of nonradial solutions related to different values of $k$ are separated.
The Rabinowitz alternative theorem holds (see \cite[Theorem 1]{DA83}) and the bifurcation is indeed global and this proves the final part of the Theorem.
\end{proof}
\begin{remark}
In some particular cases it is possible to improve the statement of Theorem \ref{a15}. For example, if $f(s)=\l e^s$ it will be showed in the next section that the function $F_k(\l)$ in \eqref{a16} is strictly increasing in $\l$ for any $k$. This allow us to compute exactly the number of nonradial bifurcation points.
\end{remark}
\begin{remark} Hypothesis ii) in Theorem \ref{a15} is satisfied if the first eigenvalue $\nu_1(\l)$ is analytic in $\l$. This is the case, for example, if $f(\l,s)$ is analytic in $\l$, see \cite{K}. Anyway, the analyticity of $f(\l,s)$ is not a necessary condition for $ii)$ to hold, and in some cases a strict monotonicity property of $\nu_1(\l)$ can be proved directly (see Proposition \ref{p-ni1}).
\end{remark}
\begin{remark}\label{R1}
As we did in Theorem \ref{a15} we can state a bifurcation result with respect to the parameter $\a$, getting the following result:\\
{\it Let $\l\in (a,b)$ be fixed; then if
\begin{equation}\label{a17}
\a=\a_k=\frac{2k}{\sqrt {-\nu_1(\lambda)}}-2>0\hbox{ with }k\in \N
\end{equation}
there exists a branch of solutions  bifurcating from $(\a_k,u_{\lambda,\a_k})$ in $(0,+\infty)\times C^{1,\g}_0(\bar B_1)$.
Moreover, branches of bifurcating solutions
related to different values of $k$ are separated and either they are 
unbounded in the space $(0,+\infty)\times C^{1,\g}_0(\bar B_1)$ or they meet $\{0\}\times C^{1,\g}_0(\bar B_1) $.}\\
These branches of solutions are obviously different from those
obtained in Theorem \ref{a15}, but do not allow to derive if there
are other pairs $(\l,\a)$ other than those previously found.
However, we think that the set ${\mathcal V_k}$
obtained in this way coincides with the set ${\mathcal
U_k}$ of the Theorem \ref{a15}.
\end{remark}
\section{The case of the exponential nonlinearity}\label{s4}
In this section we apply the results of Section \ref{s3}
to the exponential nonlinearity $f(\l,s)=\l e^s$, i.e. to the problem
\begin{equation} \label{p'}
\begin{cases}
-\Delta u = \left(\frac{2+\alpha}{2}\right)^2\l \abs{x}^{\alpha}e^u, & \hbox{ in }B_1 \\
u > 0, & \hbox{ in }B_1 \\
 u = 0, & \hbox{ on } \partial B_1
\end{cases}
\end{equation}
for $\l>0$.  Let us start by considering radial solutions to the problem \eqref{p'}.
In this case there exists a maximal value of $\l$ that separates the threshold between existence and nonexistence of solutions of \eqref{p'}.
\begin{theorem}\label{b1}
Let us consider the problem \eqref{p'}. We have that,\\
$i)$ if $\l\in\left(0,2\right)$  there exist exactly two radial solutions $u_+$ and $u_-$ given by
\begin{equation}\label{sol-exp}
u_\pm(x) = \log\frac{8\delta_{\l}^{\pm}}{\l (\delta_{\l}^{\pm}+
\abs{x}^{2+\alpha})^2 }
\end{equation}
with
\begin{equation}
\delta_\l^{\pm} = \frac{4 - \l \pm 2 \sqrt{4 - 2\l} }{\l}.
\end{equation}
The solution $u_+$ is the minimal one and $u_-$ blows up at the origin as $\l \to 0^{+}$.\\
$ii)$ If $\l=2$ there is only the solution
\begin{equation}\label{sol-exp-2}
u(x) = \log\frac{4}{ (1 +
\abs{x}^{2+\alpha})^2 }.
\end{equation}
$iii)$ There is no solution if $\l>2$.
\end{theorem}
\begin{proof}
We set $v(r) =
u(r^{\frac{2}{2+\alpha}})$, where $r=\abs{x}$. In this way, we are
led to the problem
\begin{equation}\label{b2}
\begin{cases}
- v'' -\frac{1}{r}v' = \l e^v, & \hbox{ in } 0 < r < 1 \\
v > 0, & \hbox{ in }0 < r < 1 \\
 v'(0) = 0 = v(1).
\end{cases}
\end{equation}
It is well known that the above problem admits solutions only if
$0 < \l \leq 2$, and all solutions are given by
\begin{equation}\nonumber
v(r) = \log\left(
\frac{8{\delta_\l}}{\l({\delta_\l}+r^2)^2}\right)
\qquad \hbox{where } \quad {\delta_\l}={\delta}_\l^{\pm} =
\frac{4-\l \pm \sqrt{16 -
8\l}}{\l} .
\end{equation}
The solution with
$$ {\delta}_\l^+ = \frac{4-\l + \sqrt{16 - 8\l}}{\l}$$
is the minimal solution which goes to zero uniformly as
$\l \to 0^{+}$, while the solution with
$$ {\delta}_\l^- = \frac{4-\l - \sqrt{16 - 8\l}}{\l}$$
blows up at the origin as $\l \to 0^{+}$. Turning back to \eqref{p'} by inverting the transformation $v(r) = u(r^{\frac{2}{2+\alpha}})$ we get $i)$ and \eqref{sol-exp-2}.
Finally, reasoning exactly as in the paper of Mignot and Puel \cite[Theorem 1]{MP}, we can prove that problem \eqref{p'} has a unique solution for $\l=2$ and no solutions for $\l>2$ concluding the proof.
\end{proof}
\begin{proposition}\label{stima-pohozaev-2}
Any smooth solution $u$ of \eqref{p'} must satisfy
\begin{equation}\label{stima-e^u-2}
\l \int_{B_1}|x|^{\a}e^{u_+}\, dx\leq \l \int_{B_1}|x|^{\a}e^u\, dx\leq \l \int_{B_1}|x|^{\a}e^{u_-}\, dx
\end{equation}
\end{proposition}
\begin{proof}
Note that the first inequality is trivial, being $u_+$ the minimal solution. In order to prove the other inequality we
use the well known Pohozaev identity. We have
\begin{equation}\label{pohozaev}
\frac{(2+\a)^3}4\l\int_{B_1}|x|^{\a}e^u\, dx-\frac{(2+\a)^2}2\pi\l=\frac 12 \int_{\de B_1}\left(\frac{\partial u}{\partial \nu}\right)^2\,ds.
\end{equation}
Using the Schwartz inequality we get
\begin{equation*}
\left(\,\, \int_{\de B_1}\frac{\partial u}{\partial \nu}\,ds\right)^2\leq 2\pi \int_{\de B_1}\left(\frac{\partial u}{\partial \nu}\right)^2\,ds
\end{equation*}
which turns in an equality if and only if $u$ is radial, so that $\frac{\partial u}{\partial \nu}$ is constant on $\de B_1$.
Now we integrate equation \eqref{p'} in $B_1$,  getting
\begin{equation}\label{stima-eq}
\frac 12 \int_{\de B_1}\left(\frac{\partial u}{\partial \nu}\right)^2\,ds\geq \frac 1{4\pi}\left(\,\, \int_{\de B_1}\frac{\partial u}{\partial \nu}\,ds\right)^2=\frac 1{4\pi}\left(\frac{(2+\a)^2}4\l \int_{B_1}|x|^{\a}e^u\, dx\right)^2.
\end{equation}
Inserting \eqref{stima-eq} into \eqref{pohozaev} then we obtain that the following inequality holds
\begin{equation}\label{ineq}
4\pi(2+\a)\l\int_{B_1}|x|^{\a}e^u\, dx-8\pi^2 \l \geq \frac{(2+\a)^2}4\left(\l \int_{B_1}|x|^{\a}e^u\, dx\right)^2.
\end{equation}
Observing that the inequality \eqref{ineq} becomes an equality if $u$ coincides with  the radial solutions $u_{\pm}$ of the previous theorem,
by direct computation we get that
$$\l \int_{B_1}|x|^{\a}e^{u_+}\, dx\leq \l \int_{B_1}|x|^{\a}e^u\, dx\leq \l \int_{B_1}|x|^{\a}e^{u_-}\, dx,$$
and the claim follows.
\end{proof}
In the following we refer to $u_{\l,\a}$
as the solution which
blows up as $\l\rightarrow0^+$, i.e.
\begin{equation}\label{b3}
u_{\l,\a}(r)=\log \left(
\frac{8{\delta_\l}}{\l
({\delta_\l}+r^{2+\a})^2}\right) \qquad \hbox{where } \quad
{\delta_\l} = \frac{4-\l - \sqrt{16 -
8\l}}{\l} .
\end{equation}
Using the map $r\mapsto r^\frac2{2+\a}$ we have that \eqref{b3}
becomes
\begin{equation}\label{b3a}
v_{\l}(r)=\log\frac{8{\delta_\l}}{\l({\delta_\l}+r^{2})^2}
\end{equation}
which is a nondegenerate solution of

\begin{equation}\label{b4}
\begin{cases}
-\Delta v=\l e^v, & \hbox{ in } B_1(0)\\
v=0 & \hbox{ on } \de B_1(0)
\end{cases}
\end{equation}
with Morse index $1$ for $0<\lambda<2$. So we can define  $\nu_1(\l)$ as in \eqref{a5}. Since in this case we
know explicitly the solution $v_{\l}$ we have,
\begin{equation}\label{ni-lambda-exp}
\nu_1(\l):=\inf_{\eta\in H\atop\eta\neq 0
}\frac{\int_0^1r(\eta')^2\, dr-\int_0^1
\frac{8\d_\l}{\left(\d_\l+r^2\right)^2}r \eta^2\,
dr}{\int_0^1 r^{-1}\eta^2 \, dr}.
\end{equation}

We are now in position to prove Proposition \ref{p-ni1},
%
\begin{proof}[Proof of Proposition \ref{p-ni1}: ]
Let $\tilde\psi_{1,\l}$ be the first eigenfunction corresponding
to the first eigenvalue $  \nu_1(\l)$. It solves 
\begin{equation}\label{a}
\begin{cases}
-\tilde\psi_{1,\l}''-\frac 1r \tilde\psi_{1,\l}'-\frac {8\d_\l}{\left(\d_\l+r^2\right)^2}\tilde\psi_{1,\l} = \frac{\nu_1(\l)}{r^2}\tilde\psi_{1,\l}  & \hbox{ for }r\in (0,1) \\
\tilde\psi_{1,\l}(r)>0& \hbox{ for }r\in (0,1)\\
\tilde\psi_{1,\l}(1)= 0, \quad \int_0^1r(\tilde\psi_{1,\l}')^2+\frac{\tilde\psi_{1,\l}^2}r\, dr<\infty
\end{cases}
\end{equation}
We let $\xi:=\frac{\d_\l-r^2}{\d_\l+r^2}$ and
$R(\xi):=\tilde\psi_{1,\l}(r)$. Then $R(\xi)$ solves
\begin{equation}\label{leg-eq}
\begin{cases}
(1-\xi^2)R''-2\xi R'+  \frac{\nu_1(\l)}{1-\xi^2}R+2R=0     & \hbox{ for }\xi\in\left( \frac{\d_\l-1}{\d_\l+1} ,1\right) \\
 R\left(\frac{\d_\l-1}{\d_\l+1}\right)=0.
\end{cases}
\end{equation}
Equation \eqref{leg-eq} is the classical {\em Legendre equation}
and it has
$$R(\xi)=\left( \frac{1+\xi}{1-\xi} \right)^{\frac\gamma 2}\left(\xi-\gamma\right)$$
with $\gamma^2=-\nu_1(\l)$ as a solution. Moreover, for
$\gamma=\frac{\d_\l -1}{\d_\l+1}<0$, $R(\xi)$ is strictly
positive in $( \frac{\d_\l-1}{\d_\l+1} ,1)$ and satisfies
the boundary condition. This means that $\nu_1(\l)=-\left(\frac{\d_\l -1}{\d_\l+1}\right)^2$. 
Using the value of $\d_\l$ in
\eqref{b3} we have that \eqref{ni-exp} follows straightforward.
Inverting the transformation we get that
$$\tilde\psi_{1,\l}(r)=r^{-\gamma}\left( \frac {\d_\l-r^2}{\d_\l+r^2}-\gamma\right).$$
and since $\gamma=\frac{\d_\l -1}{\d_\l+1}$ then \eqref{b5}
follows by straightforward computations. \textcolor{red}{ The uniqueness of $\nu_1(\l)$ and $\tilde\psi_{1,\l}$ follows from Lemma \ref{B3} since $\tilde\psi_{1,\l}$ satisfies $\int_0^1\frac{\tilde\psi_{1,\l}^2}{r}\, dr <\infty$.}
\end{proof}
\noindent Once we know the explicit value of $\nu_1(\l)$, not only
we can apply the results of Section \ref{s3}, but we have even
more accurate results.
%
\begin{theorem}\label{t1.5}
Let $v_{\l}$ be the unique radial solution of \eqref{i2} with
Morse index $1$ and $u_{\l,\a}(r)=v_\l(r^{\frac {2+\a}2})$.
Then $u_{\l,\a}$ is degenerate if and only if
 $\l$ and $\a$  satisfy
\begin{equation}\label{deg-exp}
\frac{2-\l}2=\frac {4k^2}{(2+\a)^2}
\end{equation}
for some integer $k\geq 1$. The solutions of the linearized
equation at the values of $\a$ and $\l$ that satisfy
\eqref{deg-exp} are given by, in polar coordinates,
\begin{equation}\label{b5a}
\psi_k(r,\theta)=r^k\frac{2(2+\a)\left(1-r^{2(2+\a)}\right)+4k(1-r^{2+\a})^2}{(2(2+\a)^2-8k^2)(1-r^{2+\a})^2+8(2+\a)^2r^{2+\a}}
(A\sin k\theta+B\cos k\theta) \quad
\end{equation}
for any constants $A,B\in\R$. Finally the Morse index of
$u_{\l,\a}$ is equal to
\begin{equation}\label{b7}
m(\l,\alpha)=
\begin{cases}
1+ 2\left[\frac{\a+2}{2} \sqrt{\frac{2-\l}2}\right] &\hbox{ if }
\,\,\, \frac{\a+2}{2}
\sqrt{\frac{2-\l}2} \not\in \N \\
 (\a+2)\sqrt{\frac{2-\l}2}  -1 &\hbox{ if } \,\,\,  \frac{\a+2}{2} \sqrt{\frac{2-\l}2}\in \N
\end{cases}
\end{equation}
and  $m(\l,\a)\rightarrow+\infty$ as $\a\rightarrow+\infty$.
\end{theorem}
\begin{proof}
By \eqref{a6} of Proposition \ref{a4} and \eqref{ni-exp} we have
that $u_{\l,\a}$ is degenerate if and only if \eqref{deg-exp}
holds. 
Moreover, as said in Proposition \ref{a4}, the solutions of the linearized equation \eqref{a7} at the degeneracy points \eqref{deg-exp}, are given by $\psi_{k}(r)=\tilde \psi_{1,\l}(r^{\frac{2+\a}2})$ multiplied by the $k$-th spherical harmonic, so that  \eqref{b5a} follows.\\
Finally, inserting \eqref{ni-exp} in \eqref{morse-general} of
Theorem \ref{a11} we get \eqref{b7}.
\end{proof}
Our next step is to apply Theorem \ref{a15} to \eqref{p'} getting the bifurcation result.
\begin{theorem}
\label{texp-bif1-bis} Let $\a>0$ be fixed and let $u_{\l,\a}$ be as
defined above. There exist $j$ values
\begin{equation}\label{b8-bis}
\l_k=2-\frac{8 k^2}{(\a+2)^2}\hbox{ for }k=1,\dots,j,\hbox{ with }
j=\begin{cases}
1+\left[\frac \a2\right] & \hbox{if } \frac \a2\notin \N\\
\frac \a2 &\hbox{ if } \frac \a2\in \N.
 \end{cases}
\end{equation}
such that $(\l_k,u_{\l_k,\a})$ is a nonradial bifurcation point
for the curve of radial solutions $u_{\l,\a}$ of \eqref{p'}.
The bifurcation is global, and the branches are separated and unbounded in $(0,2)\times C^{1,\g}_0(\bar B_1)$.\\
\end{theorem}
\begin{figure}[h!] 
\includegraphics[scale=0.6]{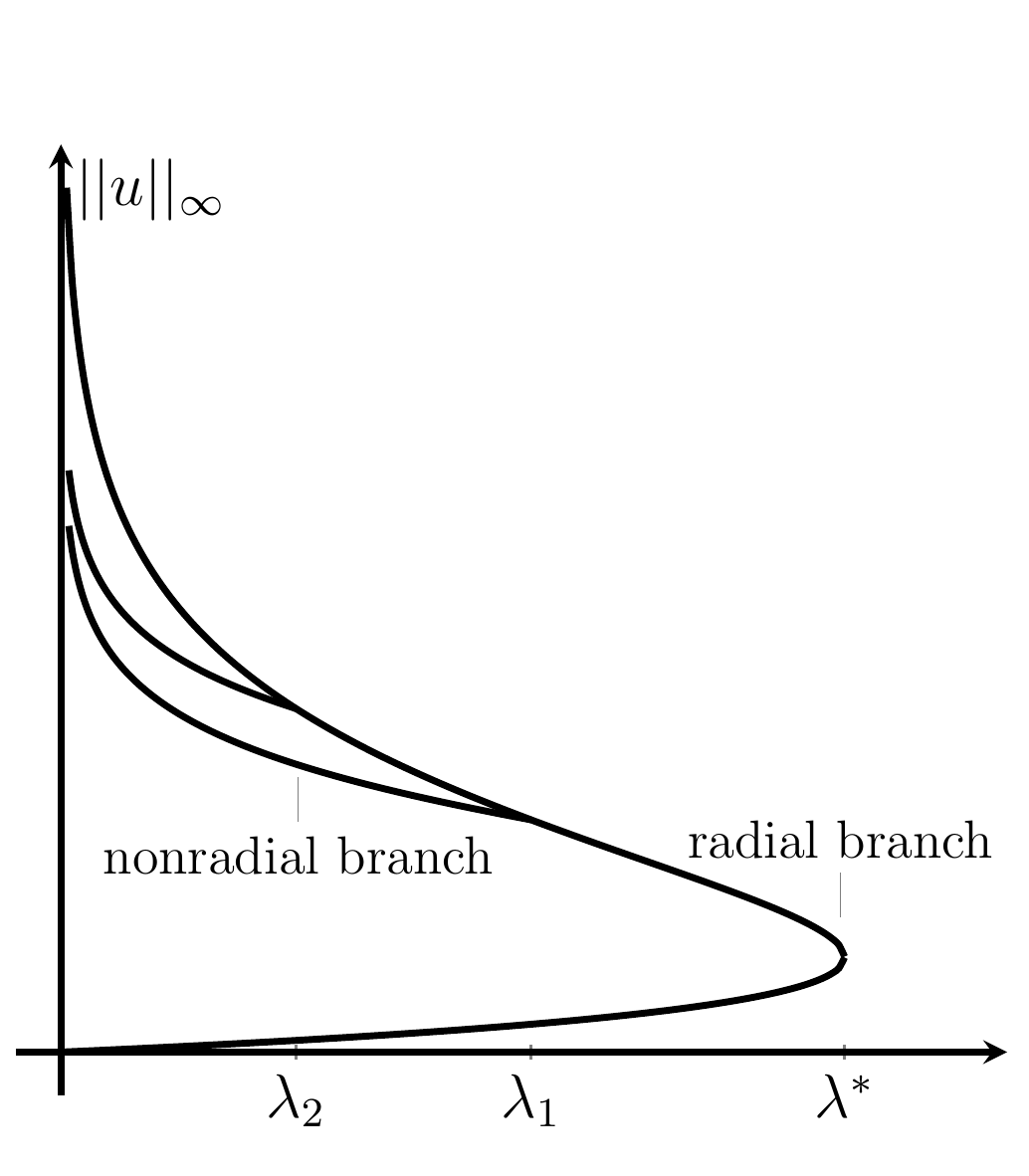}
\caption{}
\end{figure}
\begin{proof}
All the assumptions of Theorem \ref{a15} are verified and since
the equation \eqref{a6} can be explicitly solved we get the
values $\l_k$ given in \eqref{b8-bis}.
Observe that at each of the degeneracy values $\l_k$,
the Morse index of the radial solution $u_{\l,\a}$ changes.
Then  Theorem \ref{a15} implies that the bifurcation occurs
 at $(\l_k,u_{\l_k,\a})$. Moreover, setting $C(\l_k)$ as the branch of nonradial solutions bifurcating from $\l_k$, we have that it satisfies the alternative in Theorem \ref{a15}.
Observe that from \eqref{deg-exp} we have a unique bifurcation point corresponding to each value of $k\geq 1$. This implies that each branch of bifurcating solutions does not intersect the curve of radial solutions again.\\
Let us show that the branches are unbounded in
$C^{1,\g}_0(\bar B_1)$. From what we said, we only need to show that the branch $C(\l_k)$ can not stay bounded and intersect $\{0\} \times C^{1,\g}_0(\bar B_1)$ or $\{2\}\times C^{1,\g}_0(\bar B_1)$.
The case where $C(\l_k)$ is bounded and meets $\{0\} \times
C^{1,\g}_0(\bar B_1)$ cannot happen since problem \eqref{p'} has, at
$\l=0$ only the trivial solution, which is nondegenerate and
isolated.
On the other hand, the case where $C(\l_k)$ is bounded and meets $\{2\} \times
C^{1,\g}_0(\bar B_1)$ cannot happen since
problem \eqref{p'} has, when $\l=2$, a unique solution which is radially degenerate.
Then each branch $C(\l_k)$ for $k=1,\dots,j$ is unbounded.
\end{proof}
Using the result of the previous theorem we can prove a multiplicity result for problem \eqref{p'} (see Corollary \ref{c1}). To do this we need an $L^{\infty}$-estimate for solutions of \eqref{p'} when $ \l$ is bounded and bounded away from zero. This estimate can be  proved collecting the results of \cite{CL} and \cite{BCLT}  and reads as follows,
\begin{theorem}\label{t-CL}
Let $u$ be a solution of \eqref{p'} with $\a$ fixed, such that $0<c_1\leq \l\leq c_2$ for some positive constants $c_1,c_2$. Then there exists $C=C(c_1,c_2)>0$ such that
$$\nor u\nor_{L^{\infty}(B_1)}\leq C.$$
\end{theorem}
\begin{proof}
Using the assumption on $ \l$ and estimate \eqref{stima-e^u} we can apply Theorem 1.3 of \cite{BCLT} with $V(x)=\left(\frac{2+\a}{2}\right)^2 \l>0$ getting that $\sup_Ku(x)\leq C$ for any compact set $K\subset B_1$. The $L^{\infty}$-estimate near the boundary of $B_1$ follows exactly as in Theorem 1 of \cite{CL} without assuming their assumption $(i)$. Indeed, that hypothesis can be replaced by the integral estimate \eqref{stima-e^u}. This provides the boundary estimate and ends the proof.
\end{proof}
\noindent Using Theorem \ref{t-CL} then we have
\begin{corollary}\label{c1}
Let  $\a>0$ be fixed and let $\l_k$ and $j$ be as defined in \eqref{b8-bis}.
For any $\l\in (0,\l_j)$  problem \eqref{p'} has at least $j$ nonradial solutions. Moreover, for any $\l\in (\l_{k+1},\l_{k})$  problem \eqref{p'} has at least $k$ nonradial solutions.
\end{corollary}
\begin{proof}
From Theorem \ref{texp-bif1-bis} we know that there are $j$ branches that bifurcate from the radial solutions that are unbounded in $C^{1,\g}_0(\bar B_1)$. By standard regularity theory they are also unbounded in $L^{\infty}(B_1)$. By Theorem \ref{t-CL} the solutions of \eqref{p'} can blow up only as $\l \to 0^+$.
This implies that the branch bifurcating from the value $\l_k$ exists in the interval $(0,\l_k)$.
\end{proof}

Now we give the results for solutions of problem \eqref{exp}. Let us set
\begin{equation}\label{b0}
\mu=\l \left(\frac{2+\alpha}{2}\right)^2
\end{equation}
for $\mu>0$. Of course all  the previous results follow substituting $\mu=\l\left(\frac{2+\alpha}{2}\right)^2$. First, Theorem \ref{b1} becomes
\begin{theorem}\label{??}
Let us consider the problem \eqref{exp}. We have that,\\
$i)$ if $\mu\in\left(0,\frac{(2+\alpha)^2}2\right)$  there exist exactly two radial solutions $u_+$ and $u_-$ given by
\begin{equation}
u_\pm(x) = \log\left( \frac{2\delta_\mu^{\pm}(2+\alpha)^2}{\mu (\delta_\mu^{\pm}+
\abs{x}^{2+\alpha})^2 }\right)
\end{equation}
with
\begin{equation}
\delta_\mu^{\pm} = \frac{(2+\alpha)^2 - \mu \pm
(2+\alpha)\sqrt{(2+\alpha)^2 - 2\mu} }{\mu}.
\end{equation}
The solution $u_+$ is the minimal one and $u_-$ blows up as $\mu \to 0^{+}$.\\
$ii)$ If $\mu=\frac{(2+\alpha)^2}2$ there is only the solution
\begin{equation}
u(x) = \log\left( \frac{4}{ (1 +
\abs{x}^{2+\alpha})^2 }\right).
\end{equation}
$iii)$ There is no solution if $\mu>\frac{(2+\alpha)^2}2$.
\end{theorem}
Here we refer to $u_{\mu,\a}$
as the radial solution to \eqref{exp} which
blows up as $\mu\rightarrow0^+$, i.e.
\begin{equation}\label{b2b}
u_{\mu,\a}(r)=\log\left( \frac{2\delta_{\mu}(2+\alpha)^2}{\mu (\delta_{\mu}
+ \abs{x}^{2+\alpha})^2 }\right)
\end{equation}
where $\delta_{\mu}=
\frac{(2+\alpha)^2 - \mu-(2+\alpha)\sqrt{(2+\alpha)^2 -
2\mu} }{\mu}$ and $\mu\in \left(0, \frac{(2+\a)^2}2\right)$.
From Theorem \ref{t1.5} then we get
\begin{theorem}
Let $u_{\mu,\a}$ be the radial solution of \eqref{exp} defined in
\eqref{b2b}. Then $u_{\mu,\a}$ is degenerate if and only if
\begin{equation}\label{b6}
(2+\a)^2=4k^2+2\mu
\end{equation}
for some $\a>0$, $\mu\in(0,\frac{(2+\a)^2}2)$ and some integer $k\ge1$.\\
Moreover its Morse index is given by:
\begin{equation}
m(\mu,\alpha)=
\begin{cases}
1+ 2\left[ \frac1{2} \sqrt{(2+\a)^2-2\mu}\right] &\hbox{ if } \,\,\, \frac1{2} \sqrt{(2+\a)^2-2\mu} \not\in \N \\
 \sqrt{(2+\a)^2-2\mu}  -1 &\hbox{ if } \,\,\, \frac1{2} \sqrt{(2+\a)^2-2\mu} \in \N
\end{cases}
\end{equation}
and $m(\mu,\alpha)\to +\infty$ as $\a\to +\infty$.
\end{theorem}
The proof is an easy consequence of Theorem \ref{t1.5}. Finally we
only have to prove Theorems \ref{texp-bif1} and Proposition
\ref{stima-pohozaev}.
\begin{proof}[\underline{Proof of Theorem  \ref{texp-bif1}}]
It follows from Theorem \ref{texp-bif1-bis} and Corollary \ref{c1}
\end{proof}

\begin{proof}[\underline{Proof of Theorem  \ref{stima-pohozaev}}]
It follows directly from Proposition \ref{stima-pohozaev-2}.
\end{proof}

\section{Some results in $\R^2$}\label{s2}
In this section, using the transformation $r\mapsto
r^\frac{2+\a}2$, we retrieve some results, partly known and partly
new, for the problem
\begin{equation}
\label{2.1}
\begin{cases}
-\Delta u = \abs{x}^{\alpha}e^u, & \hbox{ in }\R^2 \\
\int_{\R^2} \abs{x}^{\alpha}e^u < +\infty.
\end{cases}
\end{equation}
All radial solutions to \eqref{2.1} are given by
\begin{equation} \label{2.2}
 U_{\delta,\alpha}(x) =
 \log\dfrac{2(2+\alpha)^2\delta}{(\delta+\abs{x}^{2+\alpha})^2}.
\end{equation}
We want to study the linearized problem to \eqref{2.1} at
$U_\a=U_{1,\a}$, i.e.
\begin{equation}
\label{linearized}
\begin{cases}
-\Delta v = 2(\alpha +2)^2 \frac{\abs{x}^{\alpha}}{(1+\abs{x}^{2+\alpha})^2}v & \hbox{ in } \R^2\\
\int_{\R^2}|\nabla v|^2\, dx<\infty.
\end{cases}
\end{equation}
Next theorem characterizes all solutions to \eqref{linearized}. This result
was already proved in \cite{CL1} for
$\left[\frac\a2\right]\notin\N$ and in \cite{DEM} if $\a\in2\N$.
Our proof unifies the two cases and it is (according to us)
shorter.
\begin{theorem} \label{lin}
The following alternative holds:
\begin{itemize}
\item[i)] If $\alpha\notin2\N$ the space of
solutions to \eqref{linearized} has dimension $1$ and is spanned
by
\begin{equation} \label{2.3}
v(x) = \dfrac{1-\abs{x}^{2+\alpha}}{1+\abs{x}^{2+\alpha}}.
\end{equation}
\item[ii)] If $\alpha = 2(k-1)$ for some integer $k\ge1$ the space of solutions
to \eqref{linearized} has dimension $3$ and is spanned by, in
polar coordinates,
\begin{equation} \label{2.4}
v(x) = \dfrac{1-r^{2+\alpha}}{1+r^{2+\alpha}}, \,\, v_1(x) =
\dfrac{r^k\cos(k\theta)}{1+r^{2+\alpha}} \hbox{ and } v_2(x) =
\dfrac{r^k\sin(k\theta)}{1+r^{2+\alpha}}.
\end{equation}
\end{itemize}
\end{theorem}

\begin{proof}
We decompose a solution of \eqref{linearized} using the
spherical harmonic functions, we get that $v$ is a solution of
\eqref{linearized} if and only if
$v_k(r):=\int_{S^1}v(r,\theta)Y_k(\theta)\, d\theta$ is a solution
of
 \begin{equation}\label{2.5}
\begin{cases}
-v_k''-\frac 1r v'_k+\frac {k^2}{r^2}v_k = 2\left(2+\alpha\right)^2\frac{r^{\alpha}}{(1+r^{2+\a})^2}v_k, \quad \hbox{ in }(0,+\infty) \\
v_k'(0)=0 \hbox{ if } k=0, v_k(0) = 0 \,\,\hbox {if } k \geq 1 \, \hbox { and } 
\int_0^{+\infty}r(v_k')^2\, dr<\infty
\end{cases}
\end{equation}
where $ Y_k(\theta)$ denotes a $k$-th spherical harmonic
function. Letting $\eta_k(r)=v_k(r^{\frac 2{2+\a}})$, we have that
$\eta_k$ solves
 \begin{equation}\label{2.6}
\begin{cases}
-\eta_k''-\frac 1r \eta'_k  +\frac {4k^2}{(2+\a)^2r^2}\eta_k =\frac{8}{(1+r^2)^2}\eta_k, \quad \hbox{ in }(0,+\infty) \\
\eta_k'(0)=0 \hbox{ if } k=0, \eta_k(0) = 0 \,\,\hbox {if } k \geq 1 \, \hbox { and } 
\int_0^{+\infty}r(\eta_k')^2\, dr<\infty.
\end{cases}
\end{equation}
We know that the unique solutions of \eqref{2.6} are given by
\begin{equation}\label{2.7a}
\eta_1(r)=\frac{r}{1+r^2}\text{ for } \frac {4k^2}{(2+\a)^2}=1\quad  \text{and} \quad
\eta_0(r) = \frac{1-r^2}{1+r^2} \text{ for } k=0.
\end{equation}

It follows from the Sturm comparison theorem that there are no other solutions to \eqref{2.6} besides those. Therefore, \eqref{2.6} admits a solution if, and only if, $\frac {4k^2}{(2+\a)^2} \in \{0,1\}$, which means that we must have $k =0$ or $\a = 2(k-1)$. Turning back to \eqref{2.5} we have the solutions
\begin{align*}
&v_0(r) = \frac{1-r^{2+\a}}{1+r^{2+\a}} \qquad \text{if} \:\: \a \neq 2(k-1)  \,\, \forall \,\, k \in \N\\
&v_0(r) = \frac{1-r^{2+\a}}{1+r^{2+\a}}; \,\, v_k(r) = \frac{r^{k}}{1+r^{2+\a}} \qquad \text{if} \:\: \a = 2(k-1) \,\ \text{for some} \,\, k \in \N
\end{align*}
and the proof is now complete.
\end{proof}
Next corollary computes the Morse Index of the solution $U_\a$,
extending to the case $\a\in\N$ the result in \cite{CL1}.
\begin{corollary}
Let $U_{\alpha}$ be the solution of \eqref{2.1}. Then its Morse
index is equal to
\begin{equation}
m(\alpha) =
\begin{cases}
1+ 2\left[ \frac{\alpha+2}{2} \right] &\hbox{ if } \,\,\, \frac{\alpha+2}{2} \not\in \N \\
1 + \alpha  &\hbox{ if } \,\,\, \frac{\alpha+2}{2} \in \N
\end{cases}
\end{equation}
where $[x]$ denotes the greatest integer less than or equal to $x$.
In particular we have that the Morse index of $U_{\alpha}$ changes as $\alpha$ crosses the set of even integers and also that $m(\alpha) \to \infty$ as $\alpha \to \infty$.
\end{corollary}
\begin{proof}

We have that the Morse index of $U_{\a}$ coincides with the number of negative eigenvalues (counted with their multiplicity) 
$\Lambda$ of
\begin{align}\label{2.8a}
\begin{cases}
-\Delta V - 2(2+\a)^2\frac{|x|^{\a}}{(1+|x|^{2+\a})^2}V=\Lambda V  & \hbox{ in }\R^2\\
\int_{\R^2} |\nabla V|^2\, dx <\infty .&
\end{cases}
\end{align}
Reasoning exactly as in the proof of Theorem \ref{a11} we have that the Morse index of $U_{\a}$ is given by the negative $\Lambda$ satisfying
\begin{align}\label{2.8}
\begin{cases}
-\Delta V - 2(2+\a)^2\frac{|x|^{\a}}{(1+|x|^{2+\a})^2}V=\frac{\Lambda }{|x|^2}V  & \hbox{ in }\R^2\\
 \int_{\R^2} |\nabla V|^2\, dx <\infty .&
\end{cases}
\end{align}
 Let $\Lambda < 0$ such that  \eqref{2.8} has a solution. Then  using the spherical harmonics we are led to the equation
\begin{align} \label{2.9}
\begin{cases}
-\psi_k''(r) - \frac{1}{r}\psi_k'(r) + \frac{k^2}{r^2}\psi_k(r) - 2(2+\a)^2\frac{r^\alpha}{(1+r^{2+\a})^2}\,\psi_k(r) = \frac{\Lambda}{r^2}\psi_k(r)\,, \quad \text{in} \,\,(0,\infty) \\
\psi_k'(0)=0 \hbox{ if } k=0, \psi_k(0) = 0 \hbox { if } k \geq 1, \hbox{ and } 
\int_0^{+\infty}r(\psi_k')^2\, dr<\infty
\end{cases}
\end{align}
and using the transformation $r \mapsto r^{\frac{2}{2+\a}}$, as in the previous theorem, we get
\begin{align} \label{2.10}
\begin{cases}
 -\eta_k''(r) - \frac{1}{r}\eta_k'(r)  -
 8\,\frac{\eta_k(r)}{(1+r^2)^2}=4\frac{\Lambda - k^2}{(2+\a)^2}\frac{\eta_k(r)}{r^2}\,, \quad \text{in}\,\, (0,\infty) \\
\eta_k'(0)=0 \hbox{ if } k=0, \eta_k(0) = 0 \hbox { if } k \geq 1, \hbox{ and } 
\int_0^{+\infty}r(\eta_k')^2\, dr<\infty.
\end{cases}
\end{align}
Now, by \eqref{2.6} and \eqref{2.7a}, we must have
$$ 4\frac{\Lambda - k^2}{(2+\a)^2} = -1 $$
and since $\Lambda < 0$ necessarily $k < \frac{2+\alpha}{2}$. Conversely, for each $k < \frac{2+\alpha}{2}$ we have $\Lambda = k^2 - \left(\frac{2+\alpha}{2}\right)^2 < 0$ an eigenvalue of \eqref{2.8}. Since the dimension of the eigenspace of the Laplace-Beltrami operator on $S^1$ is $2$ for any $k \ge 1$, the proof is now complete.
\end{proof}

\end{document}